\documentclass[english,11pt, a4paper]{article}

\usepackage[T1]{fontenc}

\usepackage{amssymb}
\usepackage{amsmath}
\usepackage{amsfonts}
\usepackage{amsthm}
\usepackage{mathtools}
\usepackage{mathabx}  

\usepackage[dvipsnames]{xcolor}
\usepackage{graphicx}
\usepackage{float}

\usepackage{mdframed}
\usepackage{enumitem} 

\usepackage[export]{adjustbox}
\usepackage{bbold}
\usepackage[percent]{overpic}
\usepackage{units}

\usepackage{booktabs,array}
\usepackage{url} 
\usepackage{hyperref}

\hypersetup{
    colorlinks=true,
    citecolor=purple!70!blue,
    linkcolor=black,
    filecolor=black,      
    urlcolor=black,
    pdftitle={},
    }

\usepackage{pgfplots}
\pgfplotsset{compat=1.6}

\usetikzlibrary{calc,quotes,angles,arrows.meta,positioning, shapes}
\usetikzlibrary{decorations.pathreplacing,decorations.markings}

\usepackage{float}

\usepackage{caption}
\usepackage{subcaption}

\usepackage[nocompress]{cite}

\theoremstyle{plain}%
\newtheorem{theorem}{Theorem}[section]
\newtheorem{lemma}[theorem]{Lemma}
\newtheorem{proposition}[theorem]{Proposition}

\newtheorem*{conjecture*}{Conjecture}

 \numberwithin{equation}{section}

\theoremstyle{definition}

\theoremstyle{remark}

 \let \leq \leqslant
 \let \geq \geqslant

\DeclareMathOperator{\support}{supp}

\DeclareMathOperator{\dist}{dist}

\usepackage{fancyhdr}
\usepackage{graphics}

\definecolor{detailcolor00}{rgb}{0.4405, 0.204, 0.343}
\definecolor{detailcolor01}{rgb}{0.546, 0.215, 0.352}
\definecolor{detailcolor02}{rgb}{0.675, 0.247, 0.387} 
\definecolor{detailcolor03}{rgb}{0.775, 0.317, 0.455}
\definecolor{detailcolor04}{rgb}{0.830, 0.421, 0.553} 
\definecolor{detailcolor05}{rgb}{0.831, 0.533, 0.663}
\definecolor{detailcolor06}{rgb}{0.779, 0.619, 0.775}
\definecolor{detailcolor07}{rgb}{0.724, 0.694, 0.827}
\definecolor{detailcolor08}{rgb}{0.687, 0.770, 0.880}
\definecolor{detailcolor09}{rgb}{0.671, 0.839, 0.904}
\definecolor{detailcolor10}{rgb}{0.659, 0.872, 0.882}

\usepackage{scalerel,stackengine}
\newcommand\pig[1]{\scalerel*[5.5pt]{\Big#1}{%
  \ensurestackMath{\addstackgap[1.5pt]{\big#1}}}}
\newcommand\pigl[1]{\mathopen{\pig{#1}}}
\newcommand\pigr[1]{\mathclose{\pig{#1}}}

\title{Ornstein-Zernike decay of Wilson line observables in the free phase of the \( \mathbb{Z}_2 \) lattice Higgs model}

\author{Malin P. Forsstr\"om\thanks{Chalmers University of Technology and University of Gothenburg} \thanks{Email: \url{palo@chalmers.se}}}

\begin{document}

\maketitle 

\begin{abstract} 
	In the physics literature, the Wilson line observable is believed to have a phase transition between a region with pure exponential decay and a region with Ornstein-Zernike type corrections. In~\cite{f2024b}, we confirmed the first part of this prediction. In this paper, we complement these results by showing that if \( \kappa \) is small and \( \beta \) large compared to the length of the line, then Wilson line expectations have exponential decay with Ornstein-Zernike type behavior. 
\end{abstract}

\section{Introduction}

Lattice gauge theories are a family of spin models on lattices, introduced by~\cite{w1974} as a discretization of the Yang-Mills model in physics. They were also independently introduced by Wegner in~\cite{w1974} as an example of a family of models with both local symmetries and phase transitions. In this paper, we consider the Ising lattice Higgs model, which is a lattice gauge theory, with spins in \( \mathbb{Z}_2, \) coupled to an external field which is a simple model of a Higgs field.

Let \( B_N = [-N,N]^m \cap \mathbb{Z}^m. \)
For an abelian group \( G  \) known as the stryctyre group, we let \( \Omega_1(B_N,G) \) denote the set of all \( G \)-valued 1-forms on the set \( C_1(B_N)\) of oriented edges in \( B_N,\) that is, the set of all functions \( \sigma \colon C_1(B_N) \to G \) such that for all \( e \in C_1(B_N), \) we have \( \sigma(e) = - \sigma(-e).\) In this paper, we only consider \( G = \mathbb{Z}_2, \) and in this case, we thus have \( \sigma(e) = \sigma(-e) \) for all \( e \in C_1(B_N) \) and \( \sigma \in \Omega_1(B_N,\mathbb{Z}_2). \) We let \( \rho\) be the representation of \( \mathbb{Z}_2 \) with \( \rho(0) = 1 \) and \( \rho(1) = -1 \) which maps the additive group \(\mathbb{Z}_2 \) into a multiplicative group. For \( \beta, \kappa \geq 0, \) we define the \emph{Ising lattice Higgs model} by
\begin{equation}\label{eq: model def}
	\mu_{N,\beta,\kappa}(\sigma) \coloneqq Z^{-1}_{N,\beta,\kappa} e^{\beta \sum_{p \in C_2(B_N)} \rho(d\sigma(p)) + \kappa \sum_{e \in C_1(B_N)} \rho(\sigma(e))},\quad \sigma \in \Omega^1(B_N,G),
\end{equation}
where \( Z_{N,\beta,\kappa} \) is a normalizing constant.
For local functions \( f(\sigma,) \) we let \( \mathbb{E}_{N,\beta,\kappa}[f(\sigma)] \) denote the expectation of \( f(\sigma) \) with respect to this measure, and let \( \langle f(\sigma) \rangle_{\beta,\kappa} \) to denote the limit of this expectation as \( N \to \infty. \) The existence and translation invariance of this limit is a well-known consequence of the Ginibre inequalities. For a discussion of this in the context of lattice gauge theories, see~\cite[Section 4]{flv2020}.

Throughout this paper, we let \( \gamma_n \) denote a straight path of length \( n \) with one endpoint at the origin, and let 
\begin{equation*}
	W_{\gamma_n} \coloneqq \prod_{e \in \gamma_n} \rho\bigl( \sigma(e) \bigr)
\end{equation*}
denote the corresponding \emph{Wilson line observable}. The reason that such observables are important in the lattice Higgs model is that they are believed to undergo a phase transition between a region with \emph{pure perimeter law decay} (known as the \emph{Higgs/confinement regime}), meaning that as \( n \to \infty, \)
\begin{equation*}
	\langle W_{\gamma_n} \rangle_{\beta,\kappa} \sim C_{\beta,\kappa} e^{-c_{\beta,\kappa} |\gamma_n|}
\end{equation*} 
for some non-trivial constants \( C_{\beta,\kappa} \) and \( c_{\beta,\kappa}, \) and a region with \emph{perimeter decay with polynomial corrections} (known as the \emph{free phase}), meaning that 
\begin{equation}\label{eq: OZ}
	\langle W_{\gamma_n} \rangle_{\beta,\kappa} \sim \frac{C_{\beta,\kappa} e^{-c_{\beta,\kappa} |\gamma_n|}}{p(|\gamma_n|)}
\end{equation} 
for some non-trivial constants \( C_{\beta,\kappa} \) and \( c_{\beta,\kappa}, \) and a non-constant polynomial \( p_{\beta,\kappa}(|\gamma_n|) \) (see, e.g., \cite{bf1983}). This phase transition is argued to also have a physical interpretation, corresponding to binding versus unbinding of dynamical quarks in the field of a static color source~\cite{bf1983}.
The type of decay described in~\eqref{eq: OZ} is often referred to as Ornstein-Zernike decay (see, e.g.,~\cite{civ2003}) or as exponential decay with polynomial corrections.
One reason to believe that there should be some regime of the lattice Higgs model with such decay is that when \( \beta = \infty, \) then the Wilson line expectation reduces to the spin-spin correlation of the Ising model with coupling parameter \( \kappa, \) and this model is known to undergo such a phase transition. Hence one might expect that at least for \( \beta \) large and \( \kappa \) small, we would have a similar phase.

In~\cite{f2024b}, we showed that if \( \beta \) is sufficiently small or \( \kappa \) is sufficiently large, then the expectation of Wilson line observables indeed has pure perimeter law decay. In this paper, we complement this result by showing that there are polynomial corrections in a certain dilute gas regime if \( \kappa \) is sufficiently small.  

\begin{theorem}\label{theorem: dilute free phase} 
	Assume that \( \kappa >0 \) is sufficiently small,  and let \( (\beta_n)_{n\geq 1} \) be a sequence such that  \( \lim_{n \geq 1} |\gamma_n| e^{-8(m-1)\beta_n} < \infty. \)  
	Then, there is \( C_{\beta_n,\kappa} \), defined in~\eqref{eq: C def} ad with \( 0< \liminf_{n \to \infty} C_{\beta_n,\kappa} \leq \limsup_{n \to \infty} C_{\beta_n,\kappa} < \infty \),  such that, as \( n \to \infty, \) we have
	\begin{equation}\label{eq: main result}
		\langle W_{\gamma_n} \rangle_{\beta_n,\kappa}
		\sim
		\frac{C_{\beta_n,\kappa} e^{-c_{\kappa}|\gamma_n|}}{|\gamma_n|^{\sqrt{m-1}}},
	\end{equation} 
	where \( c_\kappa \) is the same constant as that for the spin-spin correlation of the Ising model for two sites at the end-points of \( \gamma_n. \)
\end{theorem}

We now describe the main ideas of the proof. First, we do a high-temperature expansion in \( \kappa. \) The resulting quantity can be thought of as a weighted sum of Wilson loop expectations of random loops. We then do a cluster expansion for each of these random loops and show that the weight associated with loops that are much longer than \( |\gamma_n| \) is very small. This allows us to tune the parameter \( \beta_n \) so that we can approximate these Wilson loop expectations using only minimal vortices. After bounding the error terms, we obtain~\eqref{eq: main result}.

\subsection{Related papers}

Several papers from the last few years have considered dilute gas estimates in lattice gauge theories, in the sense that they have given estimates for Wilson loop or Wilson line expectations under the assumption that \( |\gamma_n| e^{-8(m-1) \beta_n} \asymp 1 \) \cite{c2019, a2021, sc2019, f2022b, flv2022, flv2020}. Closest to this paper is~\cite{a2021}, which considered Wilson loop expectations in the same setting as Theorem~\ref{theorem: dilute free phase} (also for more general structure groups). However, in contrast to Theorem~\ref{theorem: dilute free phase}, in this case, the Wilson loop expectation is shown to be of constant order, while Wilson line observables, by Theorem~\ref{theorem: dilute free phase}, have exponential decay.

In~\cite{f2024}, we studied a related observable known as the Marcu-Fredenhagen ratio and showed that this observable undergoes a phase transition. This paper also treated the free phase and did not require \( \beta_n \) to grow with the length of the loop. This result was proved using similar tools in that the proof also started by first taking a high-temperature expansion and then using a cluster expansion. However, the ideas used in~\cite{f2024} were too rough to be able to get the finer asymptotics needed to conclude a polynomial correction term to the exponential decay of Wilson line observables. Hence, one of the main contributions of this paper is the more detailed analysis of the expression resulting from the relevant cluster expansion.

If~\cite{f2022b,flv2022,f2024b}, we considered Wilson line observables in a dilute gas limit in the Higgs and confinement regimes. However, these regimes are far from the regime considered in this paper, and the ideas there do not extend to the current setting.
In particular, in~\cite{f2024b}, we showed that Wilson line observables have a pure perimeter law in the Higgs/confinement regime. Here, the observable has a different type of decay, and the methods used there do not extend to the Higgs phase of the model as the models considered there do not have finite clusters in the free phase.

\subsection{Structure of paper}

In Section~\ref{section: notation}, we review the notation we will need from discrete exterior calculus. Next, in Section~\ref{section: high temperature}, we describe the high-temperature expansion of~\eqref{eq: model def}, which is useful in the free phase, and we then in Section~\ref{section: cluster expansion} recall how cluster expansion can be used to analyze the corresponding model.  In Section~\ref{sec: bounds}, we state and prove a number of upper bounds for the cluster expansions that we will need for the proof of the main result, which is finally proven in Section~\ref{section: proof of main result}.

\subsection{Funding}

M.P.F. was funded by the Swedish Research Council, grant number 2024-04744.

\section{Background}\label{section: background}

In this section, we introduce the notation we will use throughout the paper, describe the high-temperature expansion and cluster expansion we will use, and finally, review the Ornstein-Zernike decay for the Ising model.

\subsection{Notation}\label{section: notation}

In this section, we introduce the notation we will use throughout the paper.

First, we will until the end of the proof work with the measure \( \mu_{N,\beta,\kappa} \) on the finite lattice \( B_N = [-N,N]^m \cap \mathbb{Z}^m. \) For this reason, all the notation which will be introduced in this section will depend implicitly on \( N, \) even though we usually suppress this in the notation. 

For functions \( f \) and \( g, \) will write \( f(n,N) \sim g(n,N) \) and \( f(n,N) \lesssim g(n,N) \)   to denote that \( \lim_{n \to \infty}\lim_{N \to \infty} \frac{f(n,N)}{g(n,N)} = 1 \) and that \( \lim_{n \to \infty}\lim_{N \to \infty} \frac{f(n,N)}{g(n,N)} \leq 1 \) respectively.

\subsubsection{Discrete exterior calculus}
We will use the language of discrete exterior calculus. For a thorough background to discrete exterior calculus in the setting of lattice gauge theories, we refer the reader to~\cite{flv2020}. Below, we list the notation from the discrete exterior calculus we will use in this paper.
\begin{itemize}
	\item We let \( B_N \) denote a box of sidelength \( 2N \) in \( \mathbb{Z}^m \) centered at the origin.
	
	\item For \(k =0,1,\dots, m, \) we let  \( C_k(B_N) \) denote set of oriented \( k \)-cells of \( B_N. \) 
	
	\item For \(k =0,1,\dots, m, \) we let  \( C_k(B_N,\mathbb{Z}) \) denote set of all \( \mathbb{Z} \)-valued \( k \)-chains on \( C_k (B_N). \) 
	
	\item When \( k = 1,2,\dots, m\) and  \( c \in C_k(B_N ), \) we let \( \partial c \) be the \( (k-1) \)-chain corresponding to the oriented boundary of \( c. \) When \( k = 0,1,\dots, m-1 \) and  \( c \in C_k(B_N ), \) we let \( \hat \partial c \) be the \( (k-1) \)-chain corresponding to the oriented co-boundary of \( c, \) and note that for \( c' \in C_{k+1}(B_N), \) we have \( c' \in \support \hat \partial c \Leftrightarrow c \in \support \partial c'. \)
	
	\item For \( k =1,2,\dots, m \) and  \( \mathfrak{c} \in C^k(B_N,\mathbb{Z}), \) we let \( \partial \mathfrak{c} \in C^{k-1}(B_N,\mathbb{Z})\) be defined by
	\[
	\partial \mathfrak{c}[c] \coloneqq \sum_{c' \in \partial c} \mathfrak{c}[c']
	,\quad c \in C_{k-1}(B_N).
	\] 
	 
	\item For \(k =0,1,\dots, m, \) we let \( \Omega_k(B_N,G) \) denote the set of all \( G \)-valued \( k \)-forms on \( C_k(B_N). \) When \( \mathfrak{c} \in C_k(B_N)\) is a \( k \)-chain and \( \omega \in \Omega(B_N, G), \) we write \[ \omega(\mathfrak{c}) \coloneqq \sum_{e \in C_k(B_N)^+} \mathfrak{c}[c]\omega(c). \]
	
	\item When \( k = 0,1,\dots, m-1 ,\) we let \( d \) denote the discrete differential operator which maps  \( \omega \in \Omega_k(B_N,G) \) to \( d\omega \coloneqq  \Omega_{k+1}(B_N,G) \) defined by
	\[
	d\omega(c) \coloneqq \omega(\partial c) = \sum_{c' \in \partial c} \omega(c),\quad c \in C_{k+1}(B_N). 
	\]
	
\end{itemize}

\subsubsection{Paths}
We let \( \mathcal{G}_1 \) denote the graph with vertex set \( C_1(B_N)^+ \) and an edge between two vertices \( e_1,e_2 \in C_1(B_N)^+ \) if \( \support \hat \partial e_1 \cap \support \hat \partial e_2 \neq \emptyset. \)

We say that a 1-chain \( \gamma \in C^1(B_N,\mathbb{Z}) \)  is a path if  \( \gamma(e) \in \{ -1,1, \} \) for all \( e \in \support \gamma. \)
We let \( \Lambda \) denote the set of all paths and let \( \Lambda_1 \) denote the set of all connected paths, i.e., the set of all paths \( \gamma \in \Lambda \) whose support is a connected subset of \( \mathcal{G}_1. \)
We say that a path \( \gamma \in \Lambda \) is closed if \( \partial \gamma = 0. \)

Let \( \gamma \in \Lambda_1 \) be a path. If \( \gamma \) is closed, we let \( \Lambda^\gamma \coloneqq \{ 0 \}, \) and if \( \gamma \) is not closed, then we let \( \Lambda^\gamma \) be the set of all connected paths \( \gamma_0 \) such that \( \gamma + \gamma_0 \) is closed, i.e.,
\begin{equation*}
	\Lambda^\gamma \coloneqq \bigl\{ \gamma_0 \in \Lambda_1 \colon \partial (\gamma+\gamma_0 ) = 0 \bigr\}.
\end{equation*}

When \( e \in C_1(B_N)^+ \) and \( \gamma \in \Lambda, \) we write \( e \in \gamma \) if and only if \( \gamma[e]=1, \) and when \( e \in C_1(B_N)^-, \) we write  \( e \in \gamma \) if and only if \( \gamma[-e]=-1. \)

When \( \gamma,\gamma' \in \Lambda, \) we write \( \gamma \sim \gamma' \) if there is \( e \in \gamma \) and \( e' \in \gamma' \) such that \( {\support \partial e \cap \support \hat \partial e' \neq \emptyset.} \) In other words, we write \( \gamma \sim \gamma' \) if \( \gamma \) and \( \gamma' \) both pass through some common vertex.

\subsubsection{Vortices}

We let \( \mathcal{G}_2 \) denote the graph with vertex set \( C_2(B_N)^+ \) and an edge between two vertices \( p_1,p_2 \in C_2(B_N)^+ \) if \( \support \hat \partial p_1 \cap \support \hat \partial p_2 \neq \emptyset. \) A 2-form \( \omega \in \Omega^2(B_N,\mathbb{Z}_2) \) is referred to as a \emph{vortex} if the set \( (\support \omega)^+ \) induces a connected subgraph of \( \mathcal{G}_2. \) We let \( \Lambda_2 \) denote the set of all vortices.

When \( \omega,\omega' \in \Omega^2(B_N,\mathbb{Z}_2), \) we write \( \omega \sim \gamma' \) if there is \( p \in \support \omega \) and \( p' \in \support \omega' \) such that \( \support \hat \partial p \cap \support \hat \partial p' \neq \emptyset, \) In other words, we write \( \omega \sim \omega' \) if \( \omega \) and \( \omega' \) both have support in the boundary of some common 3-cell.

One verifies that any \( \omega \in \Lambda_2 \) satisfies \( |\support \omega| \geq 2(m-1)\) (see, e.g., \cite[Figure 1]{f2022}). Any \( \omega \in \Lambda_2 \) which satisfies \( |\support \omega| = 2(m-1)\) is referred to as a minimal vortex, and can be written as \( d\sigma \) for some \( \sigma \in \Omega^1(B_N,\mathbb{Z}_2) \) with support on exactly one pair of edges \( \{ e,-e \} \subseteq  C_1(B_N). \)

\subsubsection{The abelian lattice Higgs model}

To simplify notation in the rest of the paper, for a path \( \gamma \in \Lambda_1, \) we define

\begin{equation}\label{eq: Z def}
	Z_{N,\beta,\kappa}[\gamma] \coloneqq \sum_{\sigma \in \Omega^1(B_N,\mathbb{Z}_2)} \rho \bigl(\sigma(\gamma)\bigr) e^{\beta \sum_{p \in C_2(B_N)} \rho(d\sigma(p)) + \kappa \sum_{e \in C_1(B_N)} \rho(\sigma(e))}.
\end{equation}
and note that if \( \gamma = 0, \) then \( Z_{N,\beta,\kappa}[0] = Z_{N,\beta,\kappa}. \)

\subsection{The high temperature expansion}\label{section: high temperature}

In this section, we describe the model obtained from applying a high-temperature expansion to~\eqref{eq: model def}. We refer the reader to~\cite[Lemma~5.1]{f2024} for a proof of this expansion.  

For a path \( \gamma \in \Lambda_1, \) let
\begin{equation}\label{eq: check Z}
	\check Z_{N,\beta,\kappa}[\gamma ] 
	\coloneqq   
	\sum_{\gamma_0 \in \Lambda^\gamma}  (\tanh 2\kappa)^{| \gamma_0|} 
	\check Z_{N,\beta,\kappa}[\gamma,\gamma_0 ],
\end{equation} 
where, for \( \gamma_0 \in \Lambda^\gamma, \) we let
\begin{align*}
	\check Z_{N,\beta,\kappa}[\gamma,\gamma_0 ] \coloneqq \!\!\!\!
	\sum_{\substack{\gamma' \in \Lambda \colon\\    \delta  \gamma'=0,\, \gamma' \nsim \gamma}} \!\!\!\!
	(\tanh 2\kappa)^{| \gamma'|} \!\!\!\!
	\sum_{\substack{\omega \in \Omega^2(B_N,\mathbb{Z}_2) \colon\\ d\omega=0}}   \!\!\!\! e^{\beta \sum_{p \in C_2(B_N)} (\rho(\omega(p))-1)  }    \rho(\omega(q_{ \gamma+\gamma_0}))\rho(\omega(q_{ \gamma'})) .
\end{align*}
Note that if \( \gamma \) is closed, then
\begin{equation*}
	\check Z_{N,\beta,\kappa}[\gamma] = \check Z_{N,\beta,\kappa}[\gamma,0].
\end{equation*}
By~\cite[Lemma~5.1]{f2024}, we have 
\[
    Z_{N,\beta,\kappa}[\gamma] = \frac{| \Omega^1(B_N,\mathbb{Z}_2) | (\cosh 2\kappa)^{|C_1(B_N)^+|} e^{\beta |C_2(B_N)|} }{|\{ \omega \in \Omega^2(B_N,\mathbb{Z}_2) \colon d\omega=0 \}| }   \check Z_{N,\beta,\kappa}[\gamma].
\] 
Since the fraction in the previous equation does not depend on \( \gamma, \), this gives a relationship between the model described by \( Z_{N,\beta,\kappa}[\gamma] \) and the model described by \( \check Z_{N,\beta,\kappa}[\gamma]\), known as the high-temperature expansion.
If we set \( \beta=\infty, \)  we obtain the Ising model, and in this case \( Z_{N,\infty,\kappa}[\gamma]/Z_{N,\infty,\kappa}[0] \) is exactly the spin-spin correlation between the endpoints of \( \gamma \) in an Ising model with coupling parameter \( \kappa. \) Moreover, in this case, \eqref{eq: check Z} simplifies to 
\begin{equation}\label{eq: ising expansion}
\begin{split}
	& \mathbb{E}\bigl[  \rho \bigl(\sigma( \gamma) \bigr)\bigr]_{N,\infty,\kappa}
	=
	\frac{Z_{N,\infty,\kappa}[\gamma] }{Z_{N,\infty,\kappa}[0] }
	=
	\sum_{\gamma_0 \in \Lambda^\gamma}  (\tanh 2\kappa)^{| \gamma_0|} 
	\frac{\check Z_{N,\infty,\kappa}[\gamma,\gamma_0 ]}{\check Z_{N,\infty,\kappa}[0,0 ]}
	\\&\qquad=
	\sum_{\gamma_0 \in \Lambda^\gamma}  (\tanh 2\kappa)^{| \gamma_0|} 
	\frac{
		\sum_{\substack{\gamma' \in \Lambda \colon\\    \delta  \gamma'=0}} 
		(\tanh 2\kappa)^{| \gamma'|} \mathbf{1}(\gamma' \nsim \gamma_0)
	}{
		\sum_{\substack{\gamma' \in \Lambda \colon\\    \delta  \gamma'=0}} 
		   (\tanh 2\kappa)^{| \gamma'|} }.
\end{split}	
\end{equation}
This is exactly the same high-temperature expansion that was used in, e.g.~\cite{civ2003}, to obtain Ornstein-Zernike decay for the spin-spin correlation function in the Ising model.
We can think of the model described in~\eqref{eq: ising expansion} as first picking a random path between the endpoints of \( \gamma, \) and then considering the probability that in a certain random loop model, no loop touches the random path.
We can think of the model in~\eqref{eq: check Z} adding an additional weight \( \langle W_{\gamma+\gamma_0 + \gamma' } \rangle_{N,\beta,0}\) to each pair \( ( \gamma_0, \gamma') \) in~\eqref{eq: ising expansion}. The main difficulty posed by this addition, even in a dilute gas limit, is that the size of the support of \( \gamma +\gamma_0 + \gamma'\) is almost surely  proportional to \( |C_1(B_N)| \) and hence if we do not scale \( \beta \) with \( N, \) vortices of all sizes will affect this observable.

\subsection{The cluster expansion}\label{section: cluster expansion}

In this section, we recall the cluster expansion of \( \log \bigl( \check Z_{N,\beta,\kappa}[\gamma,\gamma_0 ] / \check Z_{N,\beta,\kappa}[0] \bigr) \) from~\cite[Section 5]{f2024}. 
    To this end, recall the definitions of \( \Lambda_1, \) \( \Lambda^\gamma, \) and \( \Lambda_2 \) from the previous sections. 
    For each \( \gamma \in \Lambda_1 \), we associate a closed surface \( q_\gamma \) whose support is completely contained in \( B_N. \)
    For \( \gamma \in \Lambda_1, \) \( \gamma_0 \in \Lambda^\gamma, \) \( \gamma',\gamma'' \in \Lambda_1, \) and \( \omega,\omega' \in \Lambda_2, \) we define an interaction function \( \iota \) by
    \begin{equation*}
		\iota(\omega,\gamma') \coloneqq
 		\rho(\omega(q_{\gamma'}))   
	\end{equation*}
	\begin{equation*}
		\iota(\omega,\omega') \coloneqq \mathbf{1}(\omega \nsim \omega') 
	\end{equation*}
	and
	\begin{equation*}
		\iota(\gamma',\gamma'') \coloneqq 
	\mathbf{1}( \gamma' \nsim \gamma'')
	\end{equation*}
	and let \( \zeta \coloneqq 1-\iota. \) The action \( \phi_{\beta,\kappa} \) is defined for \( \gamma \in \Lambda_1 \) by 
	\[
	\phi_{\beta,\kappa}(\gamma) \coloneqq (\tanh 2\kappa)^{|\gamma|},
	\]
	and for \( \omega \in \Lambda_2 \) by
	\[
	\phi_{\beta,\kappa}  (\omega) \coloneqq  e^{-2\beta |\support \omega|}.
	\] 

Let \( \mathcal{G}_\Lambda \) be the graph with vertex set \( \Lambda_1\cup \Lambda_2 \) and and edge between \( \eta_1,\eta_2 \in \Lambda_1 \cup \Lambda_2 \) if \( \zeta(\eta_1,\eta_2) \neq 0. \)

Multisets of elements of \( \Lambda_1 \cup \Lambda_2 \) corresponding to connected subgraphs of \( \mathcal{G}_\Lambda \) are referred to as clusters, and the set of all such clusters is denoted by \( \Xi. \) For a cluster \( \mathcal{C} \in \Xi \) and \( \eta \in \Lambda_1 \cup \Lambda_2, \) we let \( n_{\mathcal{C}}(\eta) \) denote the multiplicity of \( \eta   \) in \( \mathcal{C}. \)

Given a cluster \( \mathcal{C} \in \Xi, \) we let \( \mathcal{C}^1 \) be the multiset \( \{ \eta \in \mathcal{C} \colon \eta \in \Lambda_1 \},\) and let \( \mathcal{C}^2 \coloneqq \mathcal{C}\smallsetminus \mathcal{C}^1. \)
We let \( \Xi^1 \coloneqq \{ \mathcal{C} \in \Xi \colon \mathcal{C}= \mathcal{C}^1\}\) and \( \Xi^1_{e} \coloneqq \{ \mathcal{C} \in \Xi^1 \colon e \in \support \mathcal{C} \}. \)
Further, we let 
\begin{equation*}
\| \mathcal{C}\|_1 \coloneqq \sum_{\gamma \in \mathcal{C}^1} n_{\mathcal{C}} (\gamma) |\support \gamma |
\end{equation*}
and
\begin{equation*}
	\|\mathcal{C}\|_2 \coloneqq \sum_{\omega\in\mathcal{C}_2} n_{\mathcal{C}}(\omega)|(\support \omega)^+|.
\end{equation*}

In the cluster expansion, a family of special functions, known as Ursell functions, appear, and hence, we now define them in the context that is relevant for us. 
    To this end, let \( {k \geq 1,}\) and let \( \mathcal{G}^k\) be the set of all connected graphs \( G\) with vertex set \( V(G) = \{ 1,2,\dots, k\}.\) Let \( E(G)\) be the (undirected) edge set of \( G.\)
    For any polymers \( {\eta_1,\eta_2,\dots , \eta_k \in \Lambda} ,\) we let
    \begin{equation*}
        U(  \eta_1, \ldots, \eta_k  ) \coloneqq \frac{1}{k!} \sum_{G \in \mathcal{G}^k} (-1)^{|E(G )|} \prod_{(i,j) \in E({G})} \zeta( \eta_i , \eta_j).
    \end{equation*} 
    Note that this definition is invariant under permutations of the polymers \( \eta_1,\eta_2,\dots,\eta_k.\)
   For \( \mathcal{C} \in \Xi,\) and any enumeration \( \eta_1,\dots, \eta_k \) (with multiplicities) of the polymers in \( \mathcal{C},\) we define
   \begin{equation}\label{eq: ursell functions} 
        U(\mathcal{S}) = k! \, U(\eta_1,\dots, \eta_1).
   \end{equation}
Note that for any \( \mathcal{C} = \{ \eta \} \in \Xi	\)  we have \( U(\mathcal{C})=1,\) and for any \( \mathcal{C} = \{ \eta_1,\eta_2 \} \in \Xi,\) we have \( U(\mathcal{C}) = -1. \)

	Next, for \( \mathcal{C} \in \Xi, \) let \( \Psi_{\beta,\kappa}(\mathcal{C}) \coloneqq U(\mathcal{C})  \phi_{\beta,\kappa}(\mathcal{C}) .\)	
	Then, using the notation of~\cite[Section 5]{f2024}, for all \( \beta>\beta^{\text{(free)}}(\alpha) \) and \( \kappa<\kappa^{\text{(free)}}(\alpha), \) with \( \beta^{\text{(free)}} \),  \( \kappa^{\text{(free)}} \), and \( \alpha \) defined below, we can write 
	\begin{align*}
	 	\log \bigl( \check Z_{N,\beta,\kappa}[\gamma,\gamma_0 ] / \check Z_{N,\beta,\kappa}[0] \bigr)
	 	=
	 	\sum_{\mathcal{C} \in \Xi} \Psi_{\beta,\kappa}(\mathcal{C}) \pigl( \rho\bigl(\mathcal{C}^2(q_{\gamma+\gamma_0})\bigr) \mathbf{1}(\mathcal{C}^1 \nsim \gamma_0)-1\pigr),
	\end{align*}
	and hence
	\begin{equation}\label{eq: the expansion}
    	\frac{Z_{N,\beta,\kappa}[\gamma]}{Z_{N,\beta,\kappa}[0]} = 
    	\sum_{\gamma_0 \in \Lambda^\gamma}  (\tanh 2\kappa)^{| \gamma_0|} 
    	e^{	 	\sum_{\mathcal{C} \in \Xi} \Psi_{\beta,\kappa}(\mathcal{C}) \bigl( \rho(\mathcal{C}^2(q_{\gamma+\gamma_0})) \mathbf{1}(\mathcal{C}^1 \nsim \gamma_0)-1\bigr)}.
	\end{equation}
	It is not at all obvious that the sum on the right-hand side of~\eqref{eq: the expansion} converges, but this is guaranteed by the following result when \( \beta \) is sufficiently large and \( \kappa \) is sufficiently small.
	\begin{proposition}[Proposition 5.8 in \cite{f2024}]\label{proposition: cluster convergence 3}
	For each \( \alpha \in (0,1), \) there are \( \beta_0^{(\textrm{free})}(\alpha)>0 \) and  \( \kappa_0^{(\textrm{free})} (\alpha)>0\)  such that the following holds.
	
	\begin{enumerate}
		\item For all \( \alpha\in (0,1) ,\) \( \beta>   \beta_0^{(\textrm{free})}(\alpha) , \)  \( \kappa < \kappa_0^{(\textrm{free})} (\alpha), \)  \( \gamma \in \Lambda_1\),  \( \gamma_0 \in \Lambda_0, \) and \( \eta \in \Lambda,\) we have 
    \begin{equation*}
        \sum_{\mathcal{C} \in \Xi \colon \eta \in \mathcal{C}} \bigl|\Psi _{\beta,\kappa} (\mathcal{C})  \rho\bigl(\mathcal{C}^2(q_{\gamma+\gamma_0})\bigr) \mathbf{1}(\mathcal{C}^1 \nsim \gamma_0)\bigr| \leq \bigl| \phi _{\beta,\kappa}(\eta)\bigr|^{1-\alpha }  .
    \end{equation*}  
    \item For all \( \alpha \in (0,1) \), \( \beta>  \beta_0^{(\textrm{free})}(\alpha), \) \( \kappa <  \kappa_0^{(\textrm{free})}(\alpha)  \), \( \gamma \in \Lambda^\gamma,\) and  \( \gamma_0 \in \Lambda_0\) we have
	\begin{equation}\label{eq: cluster expansion 3}
		\log \check Z[\gamma,\gamma_0]
		=
		\sum_{\mathcal{C}\in \Xi} \Psi _{\beta,\kappa} (\mathcal{C})  \rho\bigl(\mathcal{C}^2(q_{\gamma+\gamma_0})\bigr) \mathbf{1}(\mathcal{C}^1 \nsim \gamma_0).
	\end{equation} 
    Furthermore, the series on the right-hand side of~\eqref{eq: cluster expansion 3} is absolutely convergent.
	\end{enumerate}  
\end{proposition}

If \( \mathcal{C} \in \Xi^1, \) then \( \Psi_{\beta,\kappa}(\mathcal{C}) \) is independent of \( \beta, \) and we therefor write \( \Psi_\kappa(\mathcal{C}) \coloneqq \Psi_{\beta,\kappa}(\mathcal{C}) \) in this case.
Further, we note that all results in this section are valid also for \( \beta = \infty\), and we write \( \Psi_{\infty,\kappa}(\mathcal{C})\coloneqq \lim_{\beta\to \infty} \Psi_{\beta,\kappa}(\mathcal{C}) = \Psi_\kappa(\mathcal{C}^1) \mathbf{1}(\mathcal{C}\in \Xi^1). \)

\subsection{Ornstein-Zernike decay for the Ising model}\label{section: ising decay}

When \( \beta = \infty\) in~\eqref{eq: Z def} (and hence also~\eqref{eq: check Z} and~\eqref{eq: the expansion}), we recover the Ising model. In the high-temperature regime of the Ising model, spin-spin correlation functions are well known to have Ornstein-Zernike decay (see, e.g., \cite{civ2003}). In  detail, by~\cite[Theorem 1.1]{civ2003}, for any \( \gamma \) with \( \partial \gamma = x-y \) and any \( \kappa<\kappa_c \) (here \( \kappa_c \) is the critical value for the Ising model on \( \mathbb{Z}^m\)), we have 
\begin{equation}\label{eq: ising decay}\begin{split} 
	&\langle \rho(\eta_x)\rho(\eta_y) \rangle_\kappa = \lim_{N \to \infty}\frac{Z_{N,\beta,\kappa}[\gamma]}{Z_{N,\beta,\kappa}[0]}  \sim 
	\frac{C_\kappa  e^{-c_\kappa  |\gamma|}}{|\gamma|^{\sqrt{m-1}}}.
\end{split}\end{equation}
as \( \dist(x,y) \to \infty\) for some constants  \( C_\kappa \) and \( c_\kappa \) that depend on \(  \gamma \) only through the direction of the line between the end-points of \( \gamma. \)
Using~\eqref{eq: the expansion}, it thus follows that 
\begin{equation}\label{eq: ising decay ce}\begin{split} 
	& \lim_{N \to \infty}
    	\sum_{\gamma_0 \in \Lambda^\gamma}  (\tanh 2\kappa)^{| \gamma_0|} 
    	e^{	 	\sum_{\mathcal{C} \in \Xi^1} \Psi_\kappa(\mathcal{C}) (  \mathbf{1}(\mathcal{C}^1 \nsim \gamma_0)-1)} 
	\sim 
	\frac{C_\kappa  e^{-c_\kappa  |\gamma|}}{|\gamma|^{\sqrt{m-1}}}.
\end{split}\end{equation}

To simplify the notation in what follows, for each \( \kappa < \kappa_0^{\text{(free)}}(\alpha), \) we let \( \vartheta_{\gamma_n,\kappa} \) be the measure on \( \Lambda^{\gamma_n} \) defined by
\begin{equation}\label{eq: vartheta}
	\vartheta_{\gamma_n,\kappa}(\gamma) \coloneqq \vartheta_{N,\gamma_n,\kappa}(\gamma) \coloneqq  (\tanh 2\kappa)^{| \gamma|}  e^{ 
	 	\sum_{\mathcal{C} \in \Xi^1  } \Psi_{\infty, \kappa}(\mathcal{C}) ( \mathbf{1}(\mathcal{C} \nsim \gamma)-1)}, \quad \gamma \in \Lambda^{\gamma_n}.
\end{equation}
We note that by~\eqref{eq: the expansion}, we have
\begin{equation*}
	\vartheta_{\gamma_n,\kappa}(\Lambda^{\gamma_n}) = \sum_{\gamma \in \Lambda^{\gamma_n}} \vartheta_{\gamma_n,\infty,\kappa}(\gamma) 
	= 
	\lim_{N \to \infty}\mathbb{E}\bigl[ \rho \bigl(\sigma(\gamma_n) \bigr) \bigr]_{N,\infty,\kappa} \leq 1,
\end{equation*}
and hence both \( \vartheta_{\gamma_n,\kappa} \) is a finite measure under the assumptions of Proposition~\ref{proposition: cluster convergence 3}. Moreover, by~\eqref{eq: ising decay ce}, we have 
\begin{equation}\label{eq: ising decay measure}\begin{split} 
	& \lim_{N \to \infty}
    	\vartheta_{\gamma_n,\kappa} (\Lambda^{\gamma_n})
	\sim 
	\frac{C_\kappa  e^{-c_\kappa  |\gamma|}}{|\gamma_n|^{\sqrt{m-1}}}.
\end{split}\end{equation}

\section{Useful upper and lower bounds}\label{sec: bounds} 

In this section, we state and prove a few lemmas that will be useful in the proof of our main result, Theorem~\ref{theorem: dilute free phase}.  The first of these results, Lemma~\ref{lemma: power cluster} below, gives an upper bound for the total weight of all large clusters that have a given edge in their support when weighting each cluster according to a power of its size.

\begin{lemma}\label{lemma: power cluster}
	Let \( \alpha \in (0,1), \) \( \kappa < \kappa^{\text{(free)}}_0(\alpha), \) and let \( a \in (0,1) \) be such that \(  (\tanh 2\kappa)^a < \tanh \bigl( 2\kappa_0^{\text{(free)}}(\alpha) \bigr).\) Then, for any \( e \in C_1(B_N) \) and \( m,K >0 \), we have
	\begin{align*}
		&\sum_{\mathcal{C} \in \Xi^1_{e} \colon \| \mathcal{C}\|_1 \geq K} \bigl| \Psi_\kappa(\mathcal{C}) \bigr| \cdot\| \mathcal{C}\|_1^m  \leq (\tanh 2\kappa)^{a(1-\alpha)} \sum_{k = K}^\infty k^m (\tanh 2\kappa)^{(1-a)k} .
	\end{align*} 
\end{lemma}

\begin{proof}  
	Let \( e \in C_1(B_N) \) and \( m,K >0 .\)
	Then
	\begin{align*}
		&
		\sum_{\mathcal{C} \in \Xi^1_e \colon \| \mathcal{C}\|_1 \geq K }  \bigl| \Psi_\kappa(\mathcal{C})  \bigr| \cdot \| \mathcal{C} \|_1^m 
		\leq
		\sum_{k = K}^\infty k^m \sum_{\mathcal{C} \in \Xi^1_e\colon \| \mathcal{C}\|_1 = k }  \bigl| \Psi_\kappa(\mathcal{C})  \bigr|  
		\\&\qquad= 
		\sum_{k = K}^\infty k^m  \sum_{\mathcal{C}\in \Xi_e^1 \colon   \| \mathcal{C} \|_1 = k} |U(\mathcal{C})| |\phi_{\infty,\kappa}(\mathcal{C})|
		\\&\qquad= 
		\sum_{k = K}^\infty k^m  \sum_{\mathcal{C}\in \Xi_e^1 \colon \| \mathcal{C} \|_1 = k} |U(\mathcal{C})| |\phi_{\infty,\kappa}(\mathcal{C})|^{1-a} |\phi_{\infty,\kappa}(\mathcal{C})|^a.
	\end{align*}

	Let \( \kappa' \) be defined by \(  \tanh 2\kappa' = (\tanh 2\kappa)^a. \) Then, by assumption, we have \( \kappa' < \kappa_0^{\text{(free)}}(\alpha). \) Moreover, for any \( \mathcal{C} \in \Xi^1,\) we have
	\begin{align*}
		\phi_{\infty,\kappa}(\mathcal{C})^a =  (\tanh 2\kappa)^{a\| \mathcal{C}\|_1} =  (\tanh 2\kappa')^{\| \mathcal{C}\|_1} = \phi_{\infty,\kappa'}(\mathcal{C}).
	\end{align*}
	Using this observation, it follows that for any \( k \geq K, \) we have
	\begin{align*}
		&		
		\sum_{\mathcal{C}\in \Xi_e^1 \colon  \| \mathcal{C} \|_1 = k} |U(\mathcal{C})| |\phi_{\infty,\kappa}(\mathcal{C})|^{1-a} |\phi_{\infty,\kappa}(\mathcal{C})|^a
		=
		(\tanh 2\kappa)^{(1-a)k}  \sum_{\mathcal{C}\in \Xi_e^1 \colon   \| \mathcal{C} \|_1 = k} |U(\mathcal{C})|  |\phi_{\infty,\kappa'}(\mathcal{C})|
		\\&\qquad=
		(\tanh 2\kappa)^{(1-a)k}  \sum_{\mathcal{C}\in \Xi_e^1 \colon \| \mathcal{C} \|_1 = k} |\Psi_{\infty,\kappa'}(\mathcal{C})|
		\leq
		(\tanh 2\kappa)^{(1-a)k} \sum_{\mathcal{C}\in \Xi_e^1 } |\Psi_{\infty,\kappa'}(\mathcal{C})|.
	\end{align*} 
	Finally, we note that by Proposition~\ref{proposition: cluster convergence 3}, applied with \( \eta = 1 \cdot e \), we have
	\begin{align*}
		& \sum_{\mathcal{C} \in \Xi^1_e }  |\Psi_{\infty,\kappa'}(\mathcal{C})| 
		\leq   
		\phi_{\infty,\kappa'} (1 \cdot e)^{1-\alpha}
		=
		(\tanh 2\kappa)^{a(1-\alpha)}.
	\end{align*} 
	Combining the above equations, we obtain the desired conclusion.
\end{proof}

Recall the definition of \( \vartheta_{\gamma_n,\kappa}, \) from~\eqref{eq: vartheta}. 
The next lemma gives the total mass, with respect to the measure \( \vartheta_{\gamma_n,\kappa}, \) of all paths \( \gamma\in \Lambda^{\gamma_n}  \) of length larger that some number \( K. \) We note that, by definition, \( \bigl\{ \gamma \in \Lambda^{\gamma_n} \colon |\gamma| < |\gamma_n| \bigr\} = \emptyset,\) and hence we will only be interested in applying this lemma for \( K \geq |\gamma_n|. \)

\begin{lemma}\label{lemma: large gamma0 2}
	Let \( \alpha \in (0,1) \) and \( \kappa < \kappa_0^{\text{(free)}}(\alpha). \) Then, for all \( K >0 \), we have 
	\begin{align*}
		\vartheta_{\gamma_n,\kappa} \pigl( \bigl\{ \gamma \in \Lambda^{\gamma_n} \colon |\gamma| >K \bigr\} \pigr) \leq \sum_{j = K}^\infty (2m)^j \pigl( (\tanh 2\kappa )  e^{ 
	 	2 (\tanh 2\kappa)^{1-\alpha}} \pigr)^{j}. 
	\end{align*}
\end{lemma}

\begin{proof}
	Let \( K >0. \) By definition, for each \( \gamma \in \Lambda^{\gamma_n}, \) we have
	\begin{align*}
		&
		\vartheta_{\gamma_n,\kappa}(\gamma) \leq (\tanh 2\kappa)^{| \gamma|} \sup_{e \in C_1(B_N)}   e^{ 
	 	2 |\gamma| \sum_{\mathcal{C} \in \Xi^1_{e}  }   |\Psi_{\infty, \kappa}(\mathcal{C})|}.
	\end{align*}
	For any edge \( e \in C_1(B_N),\) by Proposition~\ref{proposition: cluster convergence 3} applied with \( \eta = 1 \cdot e,\) we have
	\begin{align*}
		& \sum_{\mathcal{C} \in \Xi^1_e }  |\Psi_\kappa(\mathcal{C})| 
		\leq   
		\phi_{\infty,\kappa} (1 \cdot e)^{1-\alpha}
		=
		(\tanh 2\kappa)^{1-\alpha}.
	\end{align*} 
	Hence, for any \( \gamma \in \Lambda^{\gamma_n}, \)
	\begin{align*}
		&
		\vartheta_{\gamma_n,\kappa}(\gamma) \leq  \pigl( (\tanh 2\kappa )  e^{ 
	 	2 (\tanh 2\kappa)^{1-\alpha}} \pigr)^{| \gamma|}.
	\end{align*}
	From this, the desired conclusion readily follows.
\end{proof}

In the next lemma, given a vertex \( v \in C_0(B_N), \) we give an upper bound of the total mass, with respect to \( \vartheta_{\gamma_n,\kappa}, \) of the set of all \( \gamma \in \Lambda^{\gamma_n} \) which passes through \( v \), i.e., all paths \( \gamma \in \Lambda^{\gamma_n} \) such that there is an edge \( e \in \gamma \) with \( v \in \partial e, \) written \( v \in \gamma. \)

\begin{lemma}\label{lemma: line intersection 2}
	Let \( \alpha \in (0,1) \) and \( \kappa < \kappa_0^{\text{(free)}}(\alpha). \) 
	Further, let \( \varepsilon>0, \) and let \( v \in C_0(B_N)\) be such that  \( \dist (v, \partial \gamma_n ) \geq \varepsilon |\gamma_n|. \) Then 
	\begin{align*}
 		& \vartheta_{\gamma_n,\kappa} \bigl( \{ \gamma \in \Lambda^{\gamma_n}\colon  v \in \gamma \} \bigr)  
 		 \lesssim \frac{C_{\kappa} e^{-c_\kappa |\gamma_n|}}{(\varepsilon(1-\varepsilon))^{\sqrt{m-1}}|\gamma_n|^{m-1}}.
 	\end{align*} 
\end{lemma}

\begin{proof} 
	To simplify notation,  assume that \( \partial \gamma_n  = y-x,\) where \( x,y \in C_0(B_N). \) 
	Since each \( \gamma \in \Lambda^{\gamma_n} \) is connected  with boundary \( \partial \gamma_n, \) it has a spanning path which starts and ends in \( \partial \gamma_n. \) Fix one such spanning path \( P_{\gamma} \) for each \( \gamma \in \Lambda^{\gamma_n}. \)
	For \( \gamma \in \Lambda^{\gamma_n}, \) if \( v \in \gamma, \)  let \( \gamma^{<v} \) be the restriction of \( \gamma \) to the support of the spanning path \( \mathcal{P}_{\gamma} \) until it first visits an edge with boundary \( v, \) and let \( \gamma^{\geq v} \coloneqq \gamma- \gamma^{<v}. \)  Using this notation, we can write 
	\begin{align*}
		&\vartheta_{\gamma_n,\kappa} \bigl( \{ \gamma \in \Lambda^{\gamma_n}\colon  v \in \gamma \} \bigr) 
		=
		\sum_{\gamma \in \Lambda^{\gamma_n}\colon v\in \gamma}\vartheta_{\gamma,\kappa}(\gamma)  
		=
		\sum_{\gamma \in \Lambda^{\gamma_n}\colon v\in \gamma}  (\tanh 2\kappa)^{| \gamma|}  e^{ \sum_{\mathcal{C} \in \Xi^1  } \Psi_{\infty, \kappa}(\mathcal{C}) ( \mathbf{1}(\mathcal{C} \nsim \gamma)-1)} 
		\\&\qquad=
		\sum_{\gamma \in \Lambda^{\gamma_n} \colon v\in \gamma}  (\tanh 2\kappa)^{| \gamma|}  e^{ -\sum_{\mathcal{C} \in \Xi^1  } \Psi_{\infty, \kappa}(\mathcal{C}) \mathbf{1}(\mathcal{C} \sim \gamma )}
		\\&\qquad=
		\sum_{\gamma \in \Lambda^{\gamma_n}\colon v\in \gamma}  (\tanh 2\kappa)^{| \gamma^{<v}|}  e^{ -\sum_{\mathcal{C} \in \Xi^1  } \Psi_{\infty, \kappa}(\mathcal{C}) \mathbf{1}(\mathcal{C} \sim \gamma^{<e} )} 
		 \\&\qquad\qquad\cdot(\tanh 2\kappa)^{| \gamma^{\geq v}|}
		 e^{ -\sum_{\mathcal{C} \in \Xi^1  } \Psi_{\infty, \kappa}(\mathcal{C}) \mathbf{1}(\mathcal{C} \sim \gamma^{\geq v} )(\mathbf{1}(\mathcal{C}\nsim \gamma^{<v}))} 
		\\&\qquad\leq
		\sum_{\gamma_1 \in \Lambda_1 \colon \partial \gamma_1 = v-x}  (\tanh 2\kappa)^{| \gamma_1|}  e^{ -\sum_{\mathcal{C} \in \Xi^1  } \Psi_{\infty, \kappa}(\mathcal{C}) \mathbf{1}(\mathcal{C} \sim \gamma_1 )} 
		 \\&\qquad\qquad\cdot\sum_{\substack{\gamma_2 \in \Lambda_1 \colon  \partial \gamma_2 = y-v,\,  \gamma_2 \nsim \gamma_1}} (\tanh 2\kappa)^{| \gamma_2|}
		 e^{ -\sum_{\mathcal{C} \in \Xi^1  } \Psi_{\infty, \kappa}(\mathcal{C}) \mathbf{1}(\mathcal{C} \sim \gamma_2 ) \mathbf{1}(\mathcal{C}\nsim \gamma_1) } .
	\end{align*} 
	Now note that, given \( \gamma_1 \in \Lambda_{1} \) such that \( \partial \gamma_1 = v-x, \) the second sum above can be interpreted as the spin-spin correlation between \( v \) and \( y \) in an Ising model with parameters \( (\kappa_e)_{e \in C_1(B_N)} \) given by \( \kappa_e = \kappa \mathbf{1}(e \notin \pm \gamma_1). \) For the Ising model, spin-spin-correlations are well known to be increasing in the parameter of each edge. Hence  
	\begin{align*}
		&\sum_{\substack{\gamma_2 \in \Lambda_1 \colon \partial \gamma_2 = y-v,\,  \gamma_2 \nsim \gamma_1}}  
		(\tanh 2\kappa)^{| \gamma_2|}
		 e^{ -\sum_{\mathcal{C} \in \Xi^1  } \Psi_{\infty, \kappa}(\mathcal{C}) \mathbf{1}(\mathcal{C} \sim \gamma_2 )(\mathbf{1}(\mathcal{C}\nsim \gamma_1))}
		 \\&\qquad\leq 
		 \sum_{\substack{\gamma_2 \in \Lambda_1 \colon \partial \gamma_2 = y-v }} 
		 (\tanh 2\kappa)^{| \gamma_2|}
		 e^{ -\sum_{\mathcal{C} \in \Xi^1  } \Psi_{\infty, \kappa}(\mathcal{C}) \mathbf{1}(\mathcal{C} \sim \gamma_2 )}.
	\end{align*}  
	Combining the previous equations, we obtain
	\begin{align*}
		&\vartheta_{\gamma_n,\kappa} \bigl( \{ \gamma \in \Lambda^{\gamma_n}\colon  v \in \gamma \} \bigr)  
		\\&\qquad\leq 
		\sum_{\gamma_1 \in \Lambda_1 \colon \partial \gamma_1 = v-x}  (\tanh 2\kappa)^{| \gamma_1|}  e^{ -\sum_{\mathcal{C} \in \Xi^1  } \Psi_{\infty, \kappa}(\mathcal{C}) \mathbf{1}(\mathcal{C} \sim \gamma_1 )} 
		 \\&\qquad\qquad\cdot\sum_{\substack{\gamma_2 \in \Lambda_1 \colon  \partial \gamma_2 = y-v }} (\tanh 2\kappa)^{| \gamma_2|}
		 e^{ -\sum_{\mathcal{C} \in \Xi^1  } \Psi_{\infty, \kappa}(\mathcal{C}) \mathbf{1}(\mathcal{C} \sim \gamma_2 )  } 
		 \\&\qquad=
		  \vartheta_{\gamma_n,\kappa} \bigl( \{ \gamma \in \Lambda^{\gamma_n^{v v}}  \} \bigr)\vartheta_{\gamma_n,\kappa} \bigl( \{ \gamma \in \Lambda^{\gamma_n^{\geq v}}  \} \bigr).
	\end{align*}
	Finally, note that since \( \dist(v,\partial \gamma_n) \geq \varepsilon, \) we have 
	\begin{equation*}
		|\gamma^{<v}||\gamma^{\geq v}| \geq \varepsilon (1-\varepsilon) |\gamma_n|^2.
	\end{equation*} 
	Using~\eqref{eq: ising decay measure}, we the desired conclusion now immediately follows. 
\end{proof}

To simplify notation, we let  \(  \xi_{\beta} \coloneqq \phi_{\beta,\kappa}(\omega_0) = e^{-4\beta \cdot 2(m-1)} \) denote the action of a minimal vortex \( \omega_0 \in \Lambda_2 \), and let \( \hat \xi_{\beta} \coloneqq \phi_{\beta,\kappa}(\omega_1) = e^{-4\beta   (4(m-1)-1)}\) denote the action of the smallest non-minimal vortex \( \omega_1 \in \Lambda_2 \) (see~\cite[Figure 1]{f2022}).

The next lemma will be used to upper bound the total weight, with respect to \( \Psi_{\beta,\kappa}, \) of the set of all clusters that either contain more than one vortex or contain a vortex that is not minimal.
\begin{lemma}\label{lemma: smallest nonminimal}
	Let \( \alpha \in (0,1), \) \( \beta > \beta_0^{\text{(free)}}(\alpha) \), and  \( \kappa <   \kappa_0^{\text{(free)}}(\alpha). \) Further, let \( a \in (0,1) \) be such that \( a\beta > \beta_0^{\text{(free)}}(\alpha).\) Then, for any \( e \in C_1(B_N), \) we have
	\begin{align*}
		\sum_{\mathcal{C} \in \Xi_{e} \colon \| \mathcal{C} \|_2 > 2(m-1)  } 
	 	\bigl| \Psi_{\beta,\kappa}(\mathcal{C}) \bigr| \leq \hat \xi_{\beta}^{1-a}  \sum_{\substack{\mathcal{C}\in \Xi_e  }} (\tanh 2\kappa)^{a(1-\alpha)}.
	\end{align*}
\end{lemma}

\begin{proof}
	Let  \( e \in C_1(B_N). \) Note that if \(  \| \mathcal{C} \|_2 > 2(m-1), \) then \(  \| \mathcal{C} \|_2 \geq 4(m-1)-2 \) (see, e.g., \cite[Figure 1]{fv2023}). Consequently,
	\begin{align*}
		&\sum_{\mathcal{C} \in \Xi_{e} \colon \| \mathcal{C} \|_2 > 2(m-1)   } 
	 	\bigl| \Psi_{\beta,\kappa}(\mathcal{C}) \bigr| 
	 	=
	 	\sum_{\mathcal{C} \in \Xi_{e} \colon \| \mathcal{C} \|_2 \geq 4(m-1)-2  } 
	 	\bigl| \Psi_{\beta,\kappa}(\mathcal{C}) \bigr| 
	 	\\&\qquad=
	 	\sum_{\mathcal{C} \in \Xi_{e} \colon \| \mathcal{C} \|_2 \geq 4(m-1)-2  } 
	 	\bigl| U(\mathcal{C}) \bigr| \cdot \bigl| \phi_{\beta,\kappa}(\mathcal{C}^2) \bigr|^{1-a}  \cdot \bigl| \phi_{\beta,\kappa}(\mathcal{C}^2) \bigr|^a \cdot \bigl| \phi_{\beta,\kappa}(\mathcal{C}^1) \bigr|
	 	\\&\qquad\leq
	 	\hat \xi_{\beta}^{1-a}  \sum_{\substack{\mathcal{C}\in \Xi_e  }} \bigl| U(\mathcal{C}) \bigr| \cdot \bigl| \phi_{\beta,\kappa}(\mathcal{C}^2) \bigr|^a \cdot \bigl| \phi_{\beta,\kappa}(\mathcal{C}^1) \bigr| 
	 	=
	 	\hat \xi_{\beta}^{1-a}  \sum_{\substack{\mathcal{C}\in \Xi_e  }} \bigl| U(\mathcal{C}) \bigr| \cdot \bigl| \phi_{a\beta,\kappa}(\mathcal{C}) \bigr|
	 	\\&\qquad=
	 	\hat \xi_{\beta}^{1-a}  \sum_{\substack{\mathcal{C}\in \Xi_e  }} \bigl| \Psi_{a\beta,\kappa}(\mathcal{C}) \bigr|.
	\end{align*}
	Next, note that by Proposition~\ref{proposition: cluster convergence 3} applied with \( \eta = 1 \cdot e, \) we have
 	\begin{align*}
		&	
		\sum_{\substack{\mathcal{C}\in \Xi_e  }} \bigl| \Psi_{a\beta,\kappa}(\mathcal{C}) \bigr|
		\leq
		\phi_{a\beta,\kappa}(1 \cdot e)^{1-\alpha} = (\tanh 2\kappa)^{a(1-\alpha)}.
 	\end{align*}
	Combining the above equations, we obtain the desired conclusion.
\end{proof}

For \( e \in C_1(B_N), \) we let 
\begin{equation*}
	\deg_{\mathcal{C}} e = \bigl| \{ \eta \in \mathcal{C}^1 \colon e \in \support \eta \} |.
\end{equation*}
 
The next lemma gives an upper bound of the total weight of all clusters \( \mathcal{C} \in \Xi^1, \) wrt. \( \Psi_\kappa, \) when weighted according to \( \deg_{\mathcal{C}} e . \)

\begin{lemma}\label{lemma: first upper bound for term 2}
	Let \( \alpha \in (0,1), \) let \( \kappa < \kappa^{\text{(free)}}_0(\alpha) ,\) and let \( a \in (0,1) \) be such that \( (\tanh 2\kappa)^a < \tanh \bigl(2\kappa_0^{\text{(free)}}(\alpha)\bigr).\) 
	Then, for any \( e \in C_1(B_N), \) we have
	\begin{equation}\label{eq: first upper bound for term 1}
		\sum_{\mathcal{C} \in \Xi^1}  \deg_{\mathcal{C}} e \cdot \bigl| \Psi_\kappa(\mathcal{C})  \bigr|  
		\leq  
		\sum_{k = 4}^\infty k (\tanh 2\kappa)^{(1-a)k} 
		(\tanh 2\kappa)^{a(1-\alpha)} \coloneqq A.
 	\end{equation}
 	Moreover,  for any \( \varepsilon>0, \) and any  \( K \geq 0 \) such that \( \frac{K(\tanh 2\kappa)^{(1-a)K}  }{(1-(\tanh 2\kappa)^{(1-a)})^2} \leq \varepsilon/4,\) we have
	\begin{align*}
 		& \vartheta_{\gamma_n,\kappa} \Bigl( \pigl\{ \gamma \in  \Lambda^{\gamma_n}\colon \sum_{e\in \gamma_n}\sum_{\mathcal{C} \in \Xi^1} \deg_\mathcal{C}(e) |\Psi_{\kappa}(\mathcal{C})|\mathbf{1}(\mathcal{C} \sim \gamma) >\varepsilon |\gamma_n| \pigr\} \Bigr) \\&\qquad\qquad\lesssim 
 		\frac{2A}{\varepsilon}  \frac{ (2K)^m C_{\kappa} e^{-c_\kappa |\gamma_n|}}{(\varepsilon/8A(1-\varepsilon/8A))^{\sqrt{m-1}}|\gamma_n|^{m-1}}. 
 	\end{align*}
\end{lemma}

\begin{proof}
	For the first statement of the proof, let \( e \in C_1(B_N), \) and note that
	\begin{align*}
		&  \sum_{\mathcal{C} \in \Xi^1}  \deg_{\mathcal{C}} e \cdot \bigl| \Psi_\kappa(\mathcal{C})  \bigr|  
		\leq 
		\sum_{\mathcal{C} \in \Xi^1_e}  \deg_{\mathcal{C}} e \cdot \bigl| \Psi_\kappa(\mathcal{C})  \bigr|  
		\leq 
		\sum_{\mathcal{C} \in \Xi^1_e}  \| \mathcal{C} \| \cdot \bigl| \Psi_\kappa(\mathcal{C})  \bigr|.  
 	\end{align*}
 	Applying Lemma~\ref{lemma: power cluster}, we obtain~\eqref{eq: first upper bound for term 1}.

 	We now prove the second statement of the lemma. To this end, note first that, given \( \gamma \in \Lambda^{\gamma_n} \) and  \( e \in \gamma_n \) such that \( \dist(e,\gamma) \geq K, \) we have
 	\begin{align*}
 		&\sum_{\mathcal{C} \in \Xi^1}  \deg_{\mathcal{C}} e \cdot \bigl| \Psi_\kappa(\mathcal{C})  \bigr|   \mathbf{1}(\mathcal{C} \sim \gamma) 
 		\leq  
 		\sum_{\mathcal{C} \in \Xi^1_e\colon \| \mathcal{C}\| \geq K} \| \mathcal{C}\|_1   \bigl|\Psi_\kappa(\mathcal{C})  \bigr| \eqqcolon B(K).
 	\end{align*}
	By Lemma~\ref{lemma: power cluster}, we have
	\begin{align*}
		&B(K) \leq (\tanh 2\kappa)^{a(1-\alpha)} \sum_{k = K}^\infty k (\tanh 2\kappa)^{(1-a)k} 
		\leq  
		\frac{K(\tanh 2\kappa)^{(1-a)K}  }{(1-(\tanh 2\kappa)^{(1-a)})^2}
		.
	\end{align*} 
	Let \( \varepsilon >0, \) and pick \( K \) large enough so that \( B(K) \leq \varepsilon/4 A.\) Further, let \( \gamma_n' \) denote the restriction of \( \gamma_n \) to the set of all edges on distance at least \( \varepsilon |\gamma_n|/8A \) from \( \partial \gamma_n. \)Then \( A|\gamma_n-\gamma_n'| = \varepsilon |\gamma_n|/4\). Let  \( \Gamma_{\gamma,K} \) denote the set of all edges in \( \gamma_n' \) that are on distance at most \( K \) from \( \gamma. \) Using this notation, we have
	\begin{align*}
 		& \vartheta_{\gamma_n,\kappa} \Bigl( \pigl\{ \gamma \in \Lambda^{\gamma_n}  \colon \sum_{e\in \gamma_n}\sum_{\mathcal{C} \in \Xi^1} \deg_\mathcal{C}(e) |\Psi_\kappa(\mathcal{C})|\mathbf{1}(\mathcal{C} \sim \gamma) >\varepsilon |\gamma_n| \pigr\} \Bigr) 
 		\\&\qquad\leq 
 		\vartheta_{\gamma_n,\kappa} \Bigl( \pigl\{ \gamma \in \Lambda^{\gamma_n} \colon B(K)|\gamma_n| + A|\gamma_n-\gamma_n'| +  A |\Gamma_{\gamma,K}|>\varepsilon |\gamma_n| \pigr\} \Bigr)  
 		\\&\qquad\leq 
 		\vartheta_{\gamma_n,\kappa} \Bigl( \pigl\{ \gamma \in \Lambda^{\gamma_n} \colon |\Gamma_{\gamma,K}|>\varepsilon |\gamma_n|/2A \pigr\} \Bigr).
	\end{align*}
	Since \( \vartheta_{\gamma_n,\kappa} \) is a positive and finite measure, we can define a probability measure \( P_{\vartheta_{\gamma_n,\kappa}} \) on \( \Lambda^{\gamma_n} \) by for \( \gamma \in \Lambda^{\gamma_n} \) letting \(  P_{\vartheta_{\gamma_n,\kappa}}(\gamma) \coloneqq \vartheta_{\gamma_n,\kappa}(\gamma)/\vartheta_{\gamma_n,\kappa}( \Lambda^{\gamma_n}).\) Using this notation, and applying Markov's inequality, we obtain 
	\begin{align*}
 		&\vartheta_{\gamma_n,\kappa} \pigl( \bigl\{ \gamma \in \Lambda^{\gamma_n} \colon |\Gamma_{\gamma,K}|>\varepsilon |\gamma_n|/2A \bigr\} \pigr)
 		\\&\qquad= 
 		\vartheta_{\gamma_n,\kappa}\bigl( \Lambda^{\gamma_n} \bigr) P_{\vartheta_{\gamma_n,\kappa}} \bigl( \gamma \in \Lambda^{\gamma_n}  \colon |\Gamma_{\gamma,K}|>\varepsilon |\gamma_n|/2A   \bigr)  
 		\\&\qquad\leq 
 		\frac{1}{\varepsilon |\gamma_n|/2A }\vartheta_{\gamma_n,\kappa}(\Lambda^{\gamma_n}) \mathbb{E}_{\vartheta_{\gamma_n,\kappa}} \bigl[   |\Gamma_{ \gamma,K}| \bigr]  
 		\\&\qquad\leq 
 		\frac{1}{\varepsilon  /2A }\vartheta_{\gamma_n,\kappa}( \Lambda^{\gamma_n} ) \max_{e \in \gamma_n'}P_{\vartheta_{\gamma_n,\kappa}} \bigl(  \dist(e,\gamma) \leq K|  \bigr). 
 		\\&\qquad= 
 		\frac{1}{\varepsilon  /2A }  \max_{e \in \gamma_n'}\vartheta_{\gamma_n,\kappa} \pigl( \bigl\{  \gamma \in \Lambda^{\gamma_n} \colon \dist(e,\gamma) \leq K|\bigr\} \pigr). 
 	\end{align*}
 	Now note that there are at most \( (2K)^m \) vertices in \( C_0(B_N)^+ \) that are on distance at most \( K \) from \( e. \) Applying Lemma~\ref{lemma: line intersection 2}, we thus obtain  
	\begin{align*}
 		&  \max_{e \in \gamma_n'}  \vartheta_{\gamma_n,\kappa} \pigl( \bigl\{  \gamma \in \Lambda^{\gamma_n} \colon \dist(e,\gamma) \leq K|\bigr\} \pigr)
 		 \lesssim   \frac{ (2K)^m C_{\kappa} e^{-c_\kappa |\gamma_n|}}{(\varepsilon/8A(1-\varepsilon/8A))^{\sqrt{m-1}}|\gamma_n|^{m-1}}. 
 	\end{align*} 
 	and hence 
 	\begin{align*}
 		&\vartheta_{\gamma_n,\kappa} \pigl( \bigl\{ \gamma \in \Lambda^{\gamma_n} \colon |\Gamma_{\gamma,K}|>\varepsilon |\gamma_n|/2A \bigr\} \pigr)
 		\lesssim
 		\frac{1}{\varepsilon/2A}  \frac{ (2K)^m C_{\kappa} e^{-c_\kappa |\gamma_n|}}{(\varepsilon/8A(1-\varepsilon/8A))^{\sqrt{m-1}}|\gamma_n|^{m-1}}.
 	\end{align*} 
 	This completes the proof.
\end{proof}

Our next lemma gives an upper bound of the contribution to \( \vartheta_{\gamma_n,\kappa} \) by \( \gamma \in \Lambda^{\gamma_n} \) that are very close to \( \gamma_n. \)

\begin{lemma}\label{lemma: upper bound on intersection}
	Let \( \alpha \in (0,1), \) \( \kappa < \kappa^{\text{(free)}}_0(\alpha) ,\) and \( 0<\varepsilon <\delta<1. \) Then
	\begin{align*}
 		& \vartheta_{\gamma_n,\kappa} \Bigl( \pigl\{ \gamma \in  \Lambda^{\gamma_n}\colon |\gamma\cap \gamma_n | > \delta|\gamma_n| \pigr\} \Bigr) 
 		\lesssim 
 		\frac{1-\varepsilon}{\delta-\varepsilon }  \cdot \frac{C_{\kappa} e^{-c_\kappa |\gamma_n|}}{(\varepsilon(1-\varepsilon))^{\sqrt{m-1}}|\gamma_n|^{m-1}} .
 	\end{align*}
\end{lemma}

\begin{proof}
	Let \( \gamma_n' \) be the restriction of \( \gamma_n \) to the set of edges on distance at least \( \varepsilon|\gamma_n|/2 \) to the endpoints of \( \gamma_n. \)
	Then
	\begin{align*} 
		\vartheta_{\gamma_n,\kappa}\pigl( \bigl\{ \gamma \in  \Lambda^{\gamma_n}\colon |\gamma\cap \gamma_n | > \delta |\gamma_n| \bigr\} \pigr) 
		\leq  
		\vartheta_{\gamma_n,\kappa}\pigl( \bigl\{ \gamma \in  \Lambda^{\gamma_n}\colon |\gamma\cap \gamma'_n | > (\delta-\varepsilon) |\gamma_n| \bigr\} \pigr) .
	\end{align*}
	Since \( \vartheta_{\gamma_n,\kappa} \) is a positive and finite measure, we can define a probability measure \( P_{\vartheta_{\gamma_n,\kappa}} \) on \( \Lambda^{\gamma_n} \) by for \( \gamma \in \Lambda^{\gamma_n} \) letting \(  P_{\vartheta_{\gamma_n,\kappa}}(\gamma) \coloneqq \vartheta_{\gamma_n,\kappa}(\gamma)/\vartheta_{\gamma_n,\kappa}( \Lambda^{\gamma_n}).\) Using this notation, and applying Markov's inequality, we obtain 
	\begin{align*}
		&
		\vartheta_{\gamma_n,\kappa}\pigl( \bigl\{ \gamma \in  \Lambda^{\gamma_n}\colon |\gamma\cap \gamma_n' | > (\delta-\varepsilon) |\gamma_n| \bigr\} \pigr) 
		=
		\vartheta_{\gamma_n,\kappa}(\Lambda^{\gamma_n})P_{\vartheta_{\gamma_n,\kappa}} \bigl( |\gamma\cap \gamma_n' | > (\delta-\varepsilon) |\gamma_n|   \bigr) 
		\\&\qquad\leq
		\frac{1}{(\delta-\varepsilon) |\gamma_n|} \vartheta_{\gamma_n,\kappa}(\Lambda^{\gamma_n})\mathbb{E}_{\vartheta_{\gamma_n,\kappa}} \bigl[  |\gamma\cap \gamma_n' |   \bigr]
		\leq
		\frac{1}{(\delta-\varepsilon) |\gamma_n|} \vartheta_{\gamma_n,\kappa}(\Lambda^{\gamma_n})   \sum_{v \in \gamma_n'} P_{\vartheta_{\gamma_n,\kappa}} ( v \in \gamma) 	
		\\&\qquad=
		\frac{1}{(\delta-\varepsilon) |\gamma_n|} \sum_{v \in \gamma_n'} \vartheta_{\gamma_n,\kappa} \bigl(\{ \gamma \in \Lambda^{\gamma_n} \colon  v \in \gamma \} \bigr) .	
	\end{align*}
	Applying Lemma~\ref{lemma: line intersection 2}, we obtain
	\begin{align*}
		&\frac{1}{(\delta-\varepsilon) |\gamma_n|}  \sum_{v \in \gamma_n'} \vartheta_{\gamma_n,\kappa} \bigl(\{ \gamma \in \Lambda^{\gamma_n} \colon  v \in \gamma \} \bigr) 
		\lesssim 
		\frac{|\gamma_n'|}{(\delta-\varepsilon) |\gamma_n|}  \cdot \frac{C_{\kappa} e^{-c_\kappa |\gamma_n|}}{(\varepsilon(1-\varepsilon))^{\sqrt{m-1}}|\gamma_n|^{m-1}}  
		\\&\qquad=
		\frac{1-\varepsilon}{\delta-\varepsilon }  \cdot \frac{C_{\kappa} e^{-c_\kappa |\gamma_n|}}{(\varepsilon(1-\varepsilon))^{\sqrt{m-1}}|\gamma_n|^{m-1}}  .
	\end{align*}
	This concludes the proof. 
\end{proof}

The next lemma essentially gives a bound on the typical size of \( \gamma \sim \vartheta_{\gamma_n,\kappa}. \)

\begin{lemma}\label{lemma: lower bound}
	Let \( \alpha \in (0,1) \),  \( \kappa < \kappa_0^{\text{(free)}}(\alpha) \),  and \( a>0. \) Then there is a constant \( K_\kappa>0 \) such that 
	\begin{equation*}
		\liminf_{n \to \infty} \vartheta_{\gamma_n,\kappa}(\Lambda^{\gamma_n})^{-1} \sum_{\gamma \in \Lambda^{\gamma_n}}   e^{ 
		-a\xi_{\beta}|\gamma|    } \vartheta_{\gamma_n,\kappa}(\gamma) \geq  e^{ 
		-a\xi_\beta K_\kappa |\gamma_n|}.
	\end{equation*}
\end{lemma}

\begin{proof}
	Let \( K \geq 1. \) Since \( \vartheta_{\gamma_n,\kappa} \) is a positive measure, we have
	\begin{align*}
		&\sum_{\gamma \in \Lambda^{\gamma_n}}   e^{ 
		-a\xi_\beta|\gamma|    } \vartheta_{\gamma_n,\kappa}(\gamma)
		\geq 
		e^{ 
		-a\xi_\beta K |\gamma_n|    }  \vartheta_{\gamma_n,\kappa} \pigl( \bigl\{ \gamma \in \Lambda^{\gamma_n} \colon |\gamma|  < K |\gamma_n| \bigr\} \pigr)
		\\&\qquad= 
		e^{ 
		-a\xi_\beta K |\gamma_n|    }  \Bigl( \vartheta_{\gamma_n,\kappa} ( \Lambda^{\gamma_n}  )-\vartheta_{\gamma_n,\kappa} \pigl( \bigl\{ \gamma \in \Lambda^{\gamma_n} \colon |\gamma| \geq K |\gamma_n| \bigr\} \pigr) \Bigr).
	\end{align*} 
	By Lemma~\ref{lemma: large gamma0 2}, we have
	\begin{align*}
		\vartheta_{\gamma_n,\kappa} \pigl( \bigl\{ \gamma \in \Lambda^{\gamma_n} \colon |\gamma|  \geq K|\gamma_n| \bigr\} \pigr) 
		\leq 
		\frac{\bigl( 2m (\tanh 2\kappa )  e^{ 
	 	2 (\tanh 2\kappa)^{1-\alpha}}  \bigr)^{K|\gamma_n|}}{1-2m(\tanh 2\kappa )  e^{ 
	 	2 (\tanh 2\kappa)^{1-\alpha}} } .
	\end{align*} 
	Since \( \vartheta_{\gamma_n,\kappa}(\Lambda^{\gamma_n}) \sim \frac{C_\kappa e^{-c_\kappa |\gamma_n|}}{|\gamma_n|^{\sqrt{m-1}}}, \) we can choose \( K \) large enough to ensure that 
	\begin{equation*}
		\bigl( 2m(\tanh 2\kappa )  e^{ 
	 	2 (\tanh 2\kappa)^{1-\alpha}} \bigr)^{K} < e^{-c_\kappa},
	\end{equation*}
	and hence such that 
	\begin{equation*}
		\limsup_{n \to \infty}\vartheta_{\gamma_n,\kappa} ( \Lambda^{\gamma_n})^{-1}\vartheta_{\gamma_n,\kappa} \pigl( \bigl\{ \gamma \in \Lambda^{\gamma_n} \colon |\gamma|  \geq K|\gamma_n| \bigr\} \pigr)  =0 .
	\end{equation*}
	From this the desired conclusion immediately follows.
\end{proof}

 Our final lemma of this section gives an upper bound to the contribution of large \( \gamma \in \Lambda^{\gamma_n} \) in the measure \( \vartheta_{\gamma_n,\kappa}(\Lambda^{\gamma_n}). \)

\begin{lemma}\label{lemma: last decay}
	Let \( \alpha \in (0,1),\) \( \kappa < \kappa_0^{\text{(free)}}(\alpha)\), and \( \delta >0, \) and assume that \( {\limsup_{n \to \infty}\xi_{\beta_n} |\gamma_n| < \infty.}\) Then
	\begin{align*}
		\lim_{n \to \infty}\vartheta_{\gamma_n,\kappa}(\Lambda^{\gamma_n})^{-1} \sum_{\gamma \in \Lambda^{\gamma_n}} \bigl(e^{ \xi_{\beta_n}^{1+\delta} |\gamma|}-1 \bigr)\vartheta_{\gamma_n,\kappa}(\gamma) = 0.
	\end{align*}
\end{lemma}

\begin{proof}
	Let \( \varepsilon>0. \) Then
	\begin{align*}
		\sum_{\gamma \in \Lambda^{\gamma_n}} \bigl(e^{ \xi_{\beta_n}^{1+\delta} |\gamma|}-1 \bigr)\vartheta_{\gamma_n,\kappa}(\gamma)
		\leq 
		\sum_{j=0}^\infty (e^{(j+1)\varepsilon}-1) \vartheta_{\gamma_n,\kappa} \pigl( \bigl\{ \gamma \in \Lambda^{\gamma_n} \colon  \xi_{\beta_n}^{1+\delta} |\gamma|  \geq j \varepsilon \bigr\} \pigr).   
	\end{align*} 
	Now let
	\begin{equation*}
		C \coloneqq 2 m  (\tanh 2\kappa )  e^{ 
	 	2 (\tanh 2\kappa)^{1-\alpha}} .
	\end{equation*}
	Since \( C < 1 \) by assumption,   for all sufficiently large \( \beta \) we have  \( e^{\varepsilon} C^{\varepsilon/(\xi_{\beta_n}^{1+\delta}) }<1 \). Using Lemma~\ref{lemma: large gamma0 2}, we can thus upper bound the previous expression by
	\begin{align*}
		&\sum_{j=0}^\infty (e^{(j+1)\varepsilon}-1) \vartheta_{\gamma_n,\kappa} \pigl( \bigl\{ \gamma \in \Lambda^{\gamma_n} \colon \xi_{\beta_n}^{1+\delta} |\gamma|  \geq j \varepsilon \bigr\} \pigr) 
		\leq 
		(e^\varepsilon-1)\vartheta_{\gamma_n,\kappa}(\Lambda^{\gamma_n}) + \sum_{j=1}^\infty e^{j\varepsilon} \frac{C^{j\varepsilon/(  \xi_{\beta_n}^{1+\delta}) }}{1-C}.
		\\&\qquad= 
		(e^\varepsilon-1)\vartheta_{\gamma_n,\kappa}(\Lambda^{\gamma_n}) + \frac{e^{\varepsilon} C^{\varepsilon/( \xi_{\beta_n}^{1+\delta}) }}{(1-C)(1-e^{\varepsilon} C^{\varepsilon/( \xi_{\beta_n}^{1+\delta}) })}.
	\end{align*}
	Since \( \limsup_{n \to \infty} \xi_{\beta_n} |\gamma_n| < \infty, \) there is \( D>0 \) such that \( \xi_{\beta_n} |\gamma_n| <D \) for all \( n \geq 1, \) and hence \( C^{\varepsilon/ \xi_{\beta_n}^{1+\delta}} 
		\leq C^{\varepsilon D^{-1}|\gamma_n|/ \xi_{\beta_n}^{\delta}} . \) Since \( \delta>0, \) for any \( a>0 \) and all sufficiently large \( \beta_n, \) we thus have \( C^{\varepsilon/ \xi_{\beta_n}^{1+\delta}}  \ll e^{-a |\gamma_n|}.  \) Since \( \varepsilon \) can be taken arbitrarily small, using~\eqref{eq: ising decay measure}, we immediately obtain the desired conclusion.
\end{proof}

\section{Proof of Theorem~\ref{theorem: dilute free phase}}
\label{section: proof of main result}

In this section, we give a proof of our main result, Theorem~\ref{theorem: dilute free phase}.

The main idea of the proof is as follows. We first recall first from Section~\ref{section: high temperature} and Section~\ref{section: cluster expansion} that we can write 
	\begin{equation}\label{eq: step 1} 
    	\frac{Z_{N,\beta,\kappa}^{(U)}[\gamma_n]}{Z_{N,\beta,\kappa}^{(U)}[0]} = 
    	\sum_{\gamma \in \Lambda^{\gamma_n}}  (\tanh 2\kappa)^{| \gamma|} 
    	e^{	 	\sum_{\mathcal{C} \in \Xi} \Psi_{\beta,\kappa}(\mathcal{C}) ( \rho(\mathcal{C}^2(q_{\gamma_n+\gamma})) \mathbf{1}(\mathcal{C}^1 \nsim \gamma)-1)}.
	\end{equation}
	and thus, we need to prove that 
	\begin{align*}
		\sum_{\gamma \in \Lambda^{\gamma_n}}  (\tanh 2\kappa)^{| \gamma|} 
    	e^{	 	\sum_{\mathcal{C} \in \Xi} \Psi_{\beta,\kappa}(\mathcal{C}) ( \rho(\mathcal{C}^2(q_{\gamma_n+\gamma})) \mathbf{1}(\mathcal{C}^1 \nsim \gamma)-1)}
    	\sim
		\frac{C_{\beta,\kappa}e^{-c_{\beta,\kappa}|\gamma_n|}}{|\gamma_n|^{\sqrt{m-1}}},
	\end{align*}
	for some constants \( c_{\beta,\kappa} \) and \( C_{\beta,\kappa}. \)
	We will aim to choose the parameter \( \beta \) large enough compared to \( |\gamma_n| \) so that, with high probability, there are vortices interacting with the random path \( \gamma_n+\gamma, \) but at the same time small enough so that with high probability no non-minimal vortices interact with \( \gamma_n+\gamma. \) We then show that under this assumption on \( \beta, \) we can compare the expression we get with that of the Ising model.

The proof of Theorem~\ref{theorem: dilute free phase} will be divided into several subsections. 

\subsection{A key lemma for the Ursell function}

The Ursell function captures the interaction between polymers in the cluster expansion. The following lemma, which is a key lemma in the proof of Theorem~\ref{theorem: dilute free phase}, allows us to factor out the contribution of the vortex polymers in the case that these are all minimal vortices with non-overlapping support.
\begin{lemma}\label{lemma: minimal ursell}
	For \( k\geq 1 \) and \( k' \geq 0, \)  let  \( \mathcal{C}  \in \Xi, \) and assume that all vortices in \( \mathcal{C}^2 \) are minimal vortices with \( \omega \nsim \omega' \) for all distinct \( \omega,\omega' \in \mathcal{C}^2. \) Then
	\begin{equation*}
		U (\mathcal{C} ) = U( \mathcal{C}^1) \cdot (-2)^{|\mathcal{C}^2|} \cdot \prod_{\omega \in \mathcal{C}^2} \deg \omega,
	\end{equation*}
	where, for \( \omega  \in \mathcal{C}^2, \) we let \( \deg \omega \coloneqq  | \{ \eta \in \mathcal{C}^1 \colon \eta \sim \omega \}|.  \)
\end{lemma}

\begin{proof}
	Since all vortices in \( \mathcal{C}^2 \) are minimal vortices that do not interact, we have \( \mathcal{C}^1 \in \Xi. \) In other words, \( \mathcal{C} \) remains a cluster even if we remove all polymers in \( \mathcal{C}^2 \) from~\( \mathcal{C}. \)

	Now fix \( \omega \in \mathcal{C}^2, \) and recall that, by assumption, \( \omega  \) is a minimal vortex. Let \( e \in C_1(B_N) \) be such that \( (\support \omega )^+ =  \support \hat \partial e. \)
	Now, assume that \( \eta,\eta' \in \mathcal{C}^1 \) are such that \( e \in \support \eta \cap \support \eta'. \)

	\begin{figure}[H]\centering
	\begin{tikzpicture}
		\fill (0,0) circle (2pt) node[left] {\( \eta' \)};
		\fill (0,2) circle (2pt) node[left] {\( \eta \)};
		\fill (2,1) circle (2pt) node[right] {\( \omega \)}; 
		\draw (0,0) -- (2,1) -- (0,2) -- (0,0);

		\draw (5.5,1) node[] {\large \( \longleftrightarrow \)};
 		
		\begin{scope}[xshift=8cm]
		\fill (0,0) circle (2pt) node[left] {\( \eta' \)};
		\fill (0,2) circle (2pt) node[left] {\( \eta \)};
		\fill (2,1) circle (2pt) node[right] {\( \omega \)}; 
		\draw (0,0) -- (2,1) -- (0,2);
		\end{scope}

	\end{tikzpicture}
	\caption{The natural bijection between graphs \( G \in \mathcal{G}^{|\mathcal{C}|} \) such that \( (\eta,\omega),(\eta',\omega),(\eta,\eta') \in E(G) \) and graphs \( G' \in \mathcal{G}^{|\mathcal{C}|} \) with \(  (\eta,\omega),(\eta',\omega)  \in E(G') \) and \( (\eta,\eta') \notin E(G') \) in the proof of Lemma~\ref{lemma: minimal ursell}.}\label{figure: graph surgery}
	\end{figure}
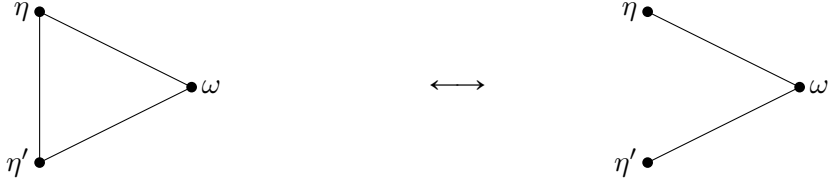
	
	Then there is a natural bijection between graphs \( G \in \mathcal{G}^{|\mathcal{C}|} \) such that \( {(\eta,\omega),(\eta',\omega),(\eta,\eta') \in E(G)} \) and graphs \( G' \in \mathcal{G}^{|\mathcal{C}|} \) with \(  (\eta,\omega),(\eta',\omega)  \in E(G') \) and \( (\eta,\eta') \notin E(G') \) (as if \(  (\eta,\omega),(\eta',\omega),(\eta,\eta')  \in E(G) \), then removing \( (\eta,\eta') \) from \( E(G) \) does not make \( G \) disconnected) (see Figure~\ref{figure: graph surgery}). Moreover, for any such pair \( G \) and \( G', \) we have
	\begin{equation*}
         (-1)^{|E(G)|} \prod_{(\eta_1,\eta_2) \in E(G)} \zeta( \eta_1 , \eta_2) =
         -   (-1)^{|E(G')|} \prod_{(\eta_1,\eta_2) \in E(G')} \zeta( \eta_1 , \eta_2) .
    \end{equation*}  
	Hence
    \begin{align*}
        &\sum_{\substack{G \in \mathcal{G}^{|\mathcal{C}|} \mathrlap{\colon}\\ (\eta,\omega),(\eta',\omega) \in E(G)}} \!\!\!\! (-1)^{|E(G )|} \prod_{(\eta_1,\eta_2) \in E(\mathcal{G})} \zeta( \eta_1 , \eta_2)   = 0.
    \end{align*} 
    This implies in particular that 
	\begin{align*}
        &U( \mathcal{C}) =  \sum_{G \in \mathcal{G}^{|\mathcal{C}|}} (-1)^{|E(G )|} \prod_{(\eta_1,\eta_2) \in E({G})} \zeta( \eta_1 , \eta_2)
         = \sum_{\substack{G \in \mathcal{G}^{|\mathcal{C}|} \mathrlap{\colon}\\ \deg(\omega) = 1\, \forall \omega \in \mathcal{C}^2}} \!\!\!\!\!\!\!\!(-1)^{|E(G )|} \prod_{\eta_1,\eta_2 \in E({G})} \zeta( \eta_1 , \eta_2). 
    \end{align*} 
    Noting that if \( \eta \in \mathcal{C}^1\) and \( \omega \in \mathcal{C}^2, \) then \( \zeta(\eta, \omega) = 2\mathbf{1}(\eta \sim \omega), \) the desired conclusion immediately follows. 
\end{proof}

\subsection{Error terms and upper bounds}

In this section, we define three error terms, \( E_1, \) \( E_2, \) and \( E_3, \)  which will appear in the proof of Theorem~\ref{theorem: dilute free phase}, and show that under the assumptions of this theorem, they are typically small.
To simplify notation in these lemmas, for \( m \geq 1, \) we define
    \begin{equation}\label{eq: def Dm}
		D_m \coloneqq \sup_{e \in C_1(B_N)}\sum_{\mathcal{C} \in \Xi^1_{e} } \bigl| \Psi_\kappa(\mathcal{C}) \bigr| \cdot\| \mathcal{C}\|^m  .
    \end{equation}
    Note that, under the assumption of Theorem~\ref{theorem: dilute free phase}, by Lemma~\ref{lemma: power cluster}, we have \( D_m < \infty ,\) and moreover, \( D_m \) can be made arvitrarily small by choosing \( \kappa \) small.

    We state and prove Lemmas~\ref{lemma: E1}--\ref{lemma: E4} under the assumptions of Theorem~\ref{theorem: dilute free phase}, without repeating those here. 

\begin{lemma}\label{lemma: E1}
	Let \( \gamma \in \Lambda^{\gamma_n}, \) and define  
	\begin{equation}
		E_1(\gamma) \coloneqq \sum_{\mathcal{C} \in \Xi \colon \| \mathcal{C} \|_2 \geq 4(m-1)-2} 
	 	\Psi_{\beta,\kappa}(\mathcal{C}) \pigl( \rho\bigl(\mathcal{C}^2(q_{\gamma_n+\gamma})\bigr) \mathbf{1}(\mathcal{C}^1 \nsim \gamma)-1\pigr).\label{eq: error 1}
	\end{equation} 
	Then 
	\begin{align*} 
	 	&|E_1(\gamma)| 
	 	\leq  
	 	2C(a,\alpha)  (|\gamma_n| + |\gamma|)    \hat \xi_{\beta}^{1-a},
	\end{align*}   
	where
	\begin{equation*}
		C(a,\alpha) \coloneqq \sum_{\substack{\mathcal{C}\in \Xi_e  }} (\tanh 2\kappa)^{a(1-\alpha)}.
	\end{equation*}
\end{lemma}

\begin{proof}
	We have
	\begin{align*} 
	 	&|E_1(\gamma)| 
	 	\leq  \sum_{\mathcal{C} \in \Xi \colon \| \mathcal{C} \|_2 \geq 4(m-1)-2} 
	 	\bigl| \Psi_{\beta,\kappa}(\mathcal{C}) \bigr| \cdot  \pigl| \rho \bigl(\mathcal{C}^2(q_{\gamma_n+\gamma})\bigr) \mathbf{1}(\mathcal{C}^1 \nsim \gamma)-1\pigr|
	 	\\&\qquad\leq 
	 	2\sum_{\mathcal{C} \in \Xi \colon \| \mathcal{C} \|_2 \geq 4(m-1)-2  } 
	 	\bigl| \Psi_{\beta,\kappa}(\mathcal{C}) \bigr| \cdot \mathbf{1}(\mathcal{C} \sim \gamma_n+ \gamma) 
	 	\\&\qquad\leq 
	 	2\bigl(|\gamma_n|+|\gamma|\bigr) \sup_{e \in C_1(B_N)} \sum_{\mathcal{C} \in \Xi_{e} \colon \| \mathcal{C} \|_2 \geq 4(m-1)-2  } 
	 	\bigl| \Psi_{\beta,\kappa}(\mathcal{C}) \bigr| .
    \end{align*}  
    Using Lemma~\ref{lemma: smallest nonminimal}, the desired conclusion immediately follows.
\end{proof}

\begin{lemma}\label{lemma: E2}
	 Let \( \gamma \in \Lambda^{\gamma_n}, \) and define
	\begin{equation*}
		E_2(\gamma) \coloneqq 4 \xi_{\beta_n} |\support \gamma_n \cap \support \gamma| .\label{eq: error 2}
	\end{equation*} 
	Then  
	\begin{align*}
 		&\lim_{n \to \infty} \vartheta_{\gamma_n,\kappa}(\Lambda_n)^{-1}\sum_{\gamma \in \Lambda_{\gamma_n}} \bigl(e^{3|E_2(\gamma)|}-1 \bigr) \, \vartheta_{\gamma_n,\kappa}(\gamma)
 		= 0.
 	\end{align*} 
\end{lemma} 

\begin{proof}
	Note first that
	\begin{align*}
		&|E_2(\gamma)| = 4 \xi_{\beta_n} |\support \gamma_n\cap \support \gamma|  \leq 4 \xi_{\beta_n} |\gamma_n|  .
	\end{align*} 
	Next, note that for any \( \delta > 0, \)  by Lemma~\ref{lemma: upper bound on intersection} (applied with \( \varepsilon = \delta/2 > 0 \)),  we have  
	\begin{align*}
 		& x{\gamma_n,\kappa} \Bigl( \pigl\{ \gamma \in  \Lambda^{\gamma_n}\colon |\support \gamma\cap \support \gamma_n | > \delta|\gamma_n| \pigr\} \Bigr) 
 		\lesssim 
 		\frac{4^{\sqrt{m-1}}}{ \delta^{\sqrt{m-1}+1} (2-\delta)^{\sqrt{m-1}-1}}  \cdot \frac{C_{\kappa} e^{-c_\kappa |\gamma_n|}}{|\gamma_n|^{m-1}} .
 	\end{align*}
 	Consequently, for any \( \epsilon>0 , \) letting \( \delta = \epsilon/(4\xi_{\beta_n}|\gamma_n|), \) it follows that 
	\begin{align*}
		&\vartheta_{\gamma_n,\kappa} \pigl( \bigl\{ \gamma \in \Lambda^{\gamma_n} \colon  |E_2(\gamma)| > \epsilon \bigr\} \pigr)
 		\leq  \vartheta_{\gamma_n,\kappa} \Bigl( \pigl\{ \gamma \in  \Lambda^{\gamma_n}\colon |\support \gamma\cap \support \gamma_n | > \epsilon/4\xi_{\beta_n} \pigr\} \Bigr)
 		\\&\qquad =  \vartheta_{\gamma_n,\kappa} \Bigl( \pigl\{ \gamma \in  \Lambda^{\gamma_n}\colon |\support \gamma\cap \support \gamma_n | > (\epsilon/(4\xi_{\beta_n}|\gamma_n|)) |\gamma_n| \pigr\} \Bigr) 
 		\\&\qquad\qquad\lesssim   
 		C(\epsilon) (\xi_{\beta_n}|\gamma_n|)^{\sqrt{m-1}+1}  \cdot \frac{C_{\kappa} e^{-c_\kappa |\gamma_n|}}{|\gamma_n|^{m-1}} \sim C(\epsilon)  \frac{(\xi_{\beta_n}|\gamma_n|)^{\sqrt{m-1}+1}}{|\gamma_n|^{\sqrt{m-1}}} \,  \vartheta_{\gamma_n,\kappa}(\Lambda^{\gamma_n}) .
 	\end{align*}
 	where
 	\begin{equation*}\label{eq: Cepsilon}
 		C(\epsilon) \coloneqq \frac{4^{\sqrt{m-1}}}{ (\epsilon/4  )^{\sqrt{m-1}+1} (2-(\epsilon/(4\xi_{\beta_n}|\gamma_n|)))^{\sqrt{m-1}-1}} \leq \frac{4^{1+2\sqrt{m-1}}}{ \epsilon^{\sqrt{m-1}+1}  }.
 	\end{equation*}
 	Finally, we note that
 	\begin{align*}
 		&\sum_{\gamma \in \Lambda_{\gamma_n}} \bigl(e^{3|E_2(\gamma)|}-1 \bigr) \, \vartheta_{\gamma_n,\kappa}(\gamma)
 		\leq e^{12\xi_{\beta_n} |\gamma_n|}\vartheta_{\gamma_n,\kappa} \pigl( \bigl\{ \gamma \in \Lambda^{\gamma_n} \colon  |E_2(\gamma)| > \epsilon \bigr\} \pigr) + (e^{3\epsilon}-1) \vartheta_{\gamma_n,\kappa}(\Lambda_n)
 		\\&\qquad\lesssim \Bigl( \frac{C(\epsilon) e^{12\xi_{\beta_n} |\gamma_n|}}{(\frac{1}{\xi_{\beta_n}|\gamma_n|})^{\sqrt{m-1}+1}  |\gamma_n|^{\sqrt{m-1}}} + (e^{3\epsilon}-1) \Bigr)\vartheta_{\gamma_n,\kappa}(\Lambda_n).
 	\end{align*}
 	Since \( \epsilon \) was arbitrary and \( 0 \ll \xi_{\beta_n} |\gamma_n| \ll \infty, \) this concludes the proof.
\end{proof}

\begin{lemma}\label{lemma: E4}
	Let \( \gamma \in \Lambda^{\gamma_n}, \) and define
	\begin{equation*}
		E_3(\gamma) \coloneqq - 4\xi_{\beta}
		\sum_{\substack{\mathcal{C} \in \Xi^1 \colon  \mathcal{C}\sim \gamma}}   \sum_{e \in \gamma_n}\Psi_\kappa(\mathcal{C})  \deg_\mathcal{C} e  \label{eq: error 4}.
	\end{equation*}
	Then
 	\begin{align*}
 		&\lim_{n \to \infty}\vartheta_{\gamma_n,\kappa}(\Lambda_n)^{-1}\sum_{\gamma \in \Lambda_{\gamma_n}} \bigl(e^{3|E_3(\gamma)|}-1 \bigr) \, \vartheta_{\gamma_n,\kappa}(\gamma)
 		=0.
 	\end{align*}
\end{lemma}

\begin{proof}
	Let \( \gamma \in \Lambda^{\gamma_n}.\) Then
	\begin{align*}
		&|E_3(\gamma)| \leq 4\xi_{\beta}
		\sum_{\substack{\mathcal{C} \in \Xi^1 \mathrlap{\colon}\\ \mathcal{C}\sim \gamma}} \;\;\sum_{e \in \gamma_n} \bigl| \Psi_\kappa(\mathcal{C}) \bigr| \deg_\mathcal{C} e 
		\leq  
		4\xi_{\beta} \sum_{e \in \gamma_n}
		\sum_{\substack{\mathcal{C} \in \Xi^1_e}} \bigl| \Psi_\kappa(\mathcal{C}) \bigr| \| \mathcal{C} \|
		\\&\qquad\leq  
		4\xi_{\beta} |\gamma_n| \sup_{e \in C_1(B_N)}
		\sum_{\substack{\mathcal{C} \in \Xi^1_{e}}} \bigl| \Psi_\kappa(\mathcal{C}) \bigr| \| \mathcal{C} \|
		= 4D_1\xi_{\beta} |\gamma_n| .
	\end{align*} 
	Next, note that by~Lemma~\ref{lemma: first upper bound for term 2}, for  \( A \coloneqq \sum_{k = 4}^\infty k (\tanh 2\kappa)^{(1-a)k} 
		(\tanh 2\kappa)^{a(1-\alpha)}, \)  \( \varepsilon >0 \) such that \( \frac{\varepsilon}{4 \xi_\beta |\gamma_n|} < 8A , \) and \( K \) such that \( \frac{K(\tanh 2\kappa)^{(1-a)K}  }{(1-(\tanh 2\kappa)^{(1-a)})^2} \leq \varepsilon/4,\) we have 
	\begin{align*}
		&\vartheta_{\gamma_n,\kappa}\pigl( \bigl\{ \gamma \in \Lambda^{\gamma_n} \colon |E_3(\gamma)| > \varepsilon \bigr\} \pigr) 
 		\\&\qquad\leq  \vartheta_{\gamma_n,\kappa} \Bigl( \pigl\{ \gamma \in \Lambda^{\gamma_n} \colon \sum_{e\in \gamma_n}\sum_{\mathcal{C} \in \Xi^1} \deg_\mathcal{C}(e) |\Psi_{\kappa}(\mathcal{C})|\mathbf{1}(\mathcal{C} \sim \gamma) >\frac{\varepsilon}{4 \xi_\beta |\gamma_n|} |\gamma_n| \pigr\} \Bigr) 
 		\\&\qquad\lesssim 
 		\frac{2A}{\frac{\varepsilon}{4 \xi_\beta |\gamma_n|}}  \frac{ (2K)^m C_{\kappa} e^{-c_\kappa |\gamma_n|}}{((\frac{\varepsilon}{32A \xi_\beta |\gamma_n|})/(1-(\frac{\varepsilon}{32A \xi_\beta |\gamma_n|})))^{\sqrt{m-1}}|\gamma_n|^{m-1}}
 		= 
 		\frac{C(\varepsilon) }{(\frac{1}{\xi_\beta |\gamma_n|})^{1+\sqrt{m-1}}|\gamma_n|^{\sqrt{m-1}}} \, \vartheta_{\gamma_n,\kappa}(\Lambda^{\gamma_n}) ,
 	\end{align*} 
 	where
 	\begin{equation*}
 		C(\varepsilon) \coloneqq \frac{8A}{ \varepsilon }  \frac{ (2K)^m }{\bigl(\frac{\varepsilon}{32A }/(1-\frac{\varepsilon}{32A \xi_\beta |\gamma_n|})\bigr)^{\sqrt{m-1}} }
 	\end{equation*} 
 	This implies in particular that
 	\begin{align*}
 		&\sum_{\gamma \in \Lambda_{\gamma_n}} \bigl(e^{3|E_3(\gamma)|}-1 \bigr) \, \vartheta_{\gamma_n,\kappa}(\gamma)
 		\leq e^{12\xi_\beta |\gamma_n|}\vartheta_{\gamma_n,\kappa} \pigl( \bigl\{ \gamma \in \Lambda^{\gamma_n} \colon  |E_3(\gamma)| > \epsilon \bigr\} \pigr) + (e^{3\epsilon}-1) \vartheta_{\gamma_n,\kappa}(\Lambda_n)
 		\\&\qquad\lesssim \Bigl( \frac{C(\varepsilon) e^{12D_1\xi_\beta |\gamma_n|}}{ (\frac{1}{\xi_\beta |\gamma_n|})^{1+\sqrt{m-1}} |\gamma_n|^{\sqrt{m-1}}} + (e^{3\varepsilon}-1) \Bigr)\vartheta_{\gamma_n,\kappa}(\Lambda_n).
 	\end{align*}
 	Since \( \varepsilon \) is arbitrary, the desired conclusion follows. 
\end{proof}

\subsection{Proof  of the main result}

In this section, we give a proof of Theorem~\ref{theorem: dilute free phase}. The first part of this proof consists of the following lemma, which allows us to rewrite the right-hand side of~\eqref{eq: step 1} in a more useful form.

\begin{lemma}\label{lemma: main contribution and errors new}
	Let \( \beta > \beta_0^{\text{(free)}}(\alpha) \) and \( \kappa < \kappa_0^{\text{(free)}}(\alpha). \) Then
	\begin{align*}
		&\sum_{\gamma \in \Lambda^{\gamma_n}}  (\tanh 2\kappa)^{| \gamma|} e^{\sum_{\mathcal{C} \in \Xi} \Psi_{\beta,\kappa}(\mathcal{C}) ( \rho(\mathcal{C}^2(q_{\gamma_n+\gamma})) \mathbf{1}(\mathcal{C}^1 \nsim \gamma)-1)} 
		\\&\qquad=
		e^{-2\xi_\beta|\gamma_n|}
		e^{
		4\xi_{\beta} \sum_{\mathcal{C} \in \Xi^1 } \sum_{e \in \gamma_n}\Psi_\kappa(\mathcal{C})  \deg_\mathcal{C} e }
		\\&\qquad\quad\cdot 
		\sum_{\gamma \in \Lambda^{\gamma_n}}   e^{ 
		-2\xi_\beta|\gamma|   
		-2\xi_\beta\sum_{\mathcal{C} \in \Xi^1 } \Psi_\kappa (\mathcal{C}) \| \mathcal{C}\| ( \mathbf{1}(\mathcal{C} \nsim \gamma)-1 )  
	 	+
	 	E_1(\gamma)+E_2(\gamma)  +E_3(\gamma)  } \vartheta_{\gamma_n,\kappa}(\gamma),
	\end{align*}  
	where \( E_1(\gamma), \) \( E_2(\gamma), \) and \( E_3(\gamma) \) are defined in~\eqref{eq: error 1},~\eqref{eq: error 2}, and~\eqref{eq: error 4} respectively.
\end{lemma}

\begin{proof} 
	Let \( \gamma \in \Lambda^{\gamma_n}. \) Then	 
	\begin{align}
	 	&\sum_{\mathcal{C} \in \Xi} \Psi_{\beta,\kappa}(\mathcal{C}) \pigl( \rho(\mathcal{C}^2 \bigl(q_{\gamma_n+\gamma})\bigr) \mathbf{1}(\mathcal{C}^1 \nsim \gamma)-1\pigr)\nonumber
	 	\\&\qquad=
	 	\sum_{\mathcal{C} \in \Xi \colon \| \mathcal{C} \|_2 \leq 2(m-1)} \Psi_{\beta,\kappa}(\mathcal{C}) \pigl( \rho \bigl(\mathcal{C}^2(q_{\gamma_n+\gamma})\bigr) \mathbf{1}(\mathcal{C}^1 \nsim \gamma)-1\pigr) 
	 	+
	 	E_1(\gamma) \label{eq: second line of firt eq new},
	\end{align}
	where \( E_1(\gamma) \) is defined in~\eqref{eq: error 1}.
	Now note that if \( \mathcal{C} \in \Xi \) satisfies \( \| \mathcal{C}\|_2 \leq 2(m-1), \) then exactly one of the following holds. 
	\begin{enumerate}[label=(\roman*), itemsep=0pt,topsep=0pt]
		\item \( \mathcal{C}^1 = \emptyset \) and \( \mathcal{C}^2 \) consists of exactly one vortex which is a minimal vortex. 	\label{eq: main case 1}	
		\item  \( \| \mathcal{C} \|_2 = 0 \), and hence  \( \mathcal{C} \in \Xi^1. \) \label{eq: main case 2}	
		\item \( \mathcal{C}^1 \neq \emptyset \) and \( \mathcal{C}^2 \) consists of exactly one vortex which is a minimal vortex.\label{eq: main case 3}	
	\end{enumerate}
	Using these cases, we can rewrite the sum in~\eqref{eq: second line of firt eq new} as
	\begin{align*}
		&\sum_{\mathcal{C} \in \Xi \colon \| \mathcal{C} \|_2 \leq 2(m-1)} \Psi_{\beta,\kappa}(\mathcal{C}) \pigl( \rho\bigl(\mathcal{C}^2(q_{\gamma_n+\gamma})\bigr) \mathbf{1}(\mathcal{C}^1 \nsim \gamma)-1\pigr)\nonumber
		\\&\qquad= 
		\sum_{\mathcal{C} \in \Xi \colon \| \mathcal{C} \|_2 \leq 2(m-1)} \Psi_{\beta,\kappa}(\mathcal{C}) \pigl( \rho \bigl(\mathcal{C}^2(q_{\gamma_n+\gamma})\bigr) \mathbf{1}(\mathcal{C}^1 \nsim \gamma)-1\pigr) \cdot \mathbf{1}(\mathcal{C}^1 = \emptyset,\, n(\mathcal{C}^2)=1)
		\\&\qquad\qquad+
		\sum_{\mathcal{C} \in \Xi \colon \| \mathcal{C} \|_2 \leq 2(m-1)} \Psi_{\beta,\kappa}(\mathcal{C}) \pigl( \rho \bigl(\mathcal{C}^2(q_{\gamma_n+\gamma})\bigr) \mathbf{1}(\mathcal{C}^1 \nsim \gamma)-1\pigr) \cdot \mathbf{1}(\mathcal{C}^2 = \emptyset) 
		\\&\qquad\qquad+
		\sum_{\mathcal{C} \in \Xi \colon \| \mathcal{C} \|_2 \leq 2(m-1)} \Psi_{\beta,\kappa}(\mathcal{C}) \pigl( \rho \bigl(\mathcal{C}^2(q_{\gamma_n+\gamma})\bigr) \mathbf{1}(\mathcal{C}^1 \nsim \gamma)-1\pigr) \cdot \mathbf{1}(\mathcal{C}^1 \neq \emptyset,\, n(\mathcal{C}^2)=1).
	\end{align*}

	We now treat these cases separately. To this end, assume first that we are in case~\ref{eq: main case 1}, i.e., assume that \( \mathcal{C}^1 = \emptyset \) and that \( \mathcal{C}^2 \) consists of exactly one vortex, which is a minimal vortex around some edge \( e. \) In this, case, we have
	\begin{align*}
		\Psi_{\beta,\kappa}(\mathcal{C}) \pigl( \rho \bigl(\mathcal{C}^2(q_{\gamma_n+\gamma})\bigr) \mathbf{1}(\mathcal{C}^1 \nsim \gamma)-1\pigr) 
		=
		-2 \xi_\beta \mathbf{1}(e \in \gamma_n +\gamma),
	\end{align*}
	and hence
	\begin{align}
		&\sum_{\mathcal{C} \in \Xi \colon \| \mathcal{C} \|_2 \leq 2(m-1)} \Psi_{\beta,\kappa}(\mathcal{C}) \pigl( \rho \bigl(\mathcal{C}^2(q_{\gamma_n+\gamma})\bigr) \mathbf{1}(\mathcal{C}^1 \nsim \gamma)-1\pigr) \cdot \mathbf{1}(\mathcal{C}^1 = \emptyset,\, n(\mathcal{C}^2)=1)\nonumber
		\\&\qquad=
		-2\xi_\beta |\gamma+\gamma_n| 
		=
		-2\xi_\beta|\gamma_n|-2\xi_\beta|\gamma| +E_2(\gamma),\label{eq: from case i new}
	\end{align}
	where \( E_2(\gamma) \) is defined in~\eqref{eq: error 2}.

	Next, assume that we are in case~\ref{eq: main case 2}, i.e., assume that \( \mathcal{C}^2 = \emptyset. \) In this case, we have 
	\begin{align*}
		&\Psi_{\beta,\kappa}(\mathcal{C}) \pigl( \rho \bigl(\mathcal{C}^2(q_{\gamma_n+\gamma})\bigr) \mathbf{1}(\mathcal{C}^1 \nsim \gamma)-1\pigr) \cdot \mathbf{1}(\mathcal{C}^2 = \emptyset) 
		=
		\Psi_{\kappa}(\mathcal{C}^2) \pigl(  \mathbf{1}(\mathcal{C}^1 \sim \gamma)-1\pigr) \cdot \mathbf{1}(\mathcal{C}^2 = \emptyset) 
	\end{align*}
	and hence
	\begin{align}
		&\sum_{\mathcal{C} \in \Xi \colon \| \mathcal{C} \|_2 \leq 2(m-1)} \Psi_{\beta,\kappa}(\mathcal{C}) \pigl( \rho \bigl(\mathcal{C}^2(q_{\gamma_n+\gamma})\bigr) \mathbf{1}(\mathcal{C}^1 \nsim \gamma)-1\pigr) \cdot \mathbf{1}(\mathcal{C}^2 = \emptyset) \nonumber
		\\&\qquad=
		\sum_{\mathcal{C} \in \Xi^1} \Psi_{\kappa}(\mathcal{C})  \bigl( \mathbf{1}(\mathcal{C}^1 \nsim \gamma)-1\bigr)  \label{eq: to combine new}.
	\end{align}
	
	Finally, assume that we are in case~\ref{eq: main case 3}, i.e. assume that \( \mathcal{C}^2 \neq \emptyset \) and that \( \mathcal{C}^2 \) consists of exactly one vortex that is a minimal vortex around some edge \( e \in C_1(B_N).\) In this case, we have \( \mathcal{C}^1 \in \Xi \), and by Lemma~\ref{lemma: minimal ursell}, we have
		\begin{equation*}
			U(\mathcal{\mathcal{C}}) = -2 U(\mathcal{\mathcal{C}}^1) \deg_{\mathcal{C}^1} e,
		\end{equation*}
	Hence
	\[
		\Psi_{\beta,\kappa}(\mathcal{\mathcal{C}}) = -2\xi_{\beta}\Psi_{\kappa}(\mathcal{\mathcal{C}}^1) \deg_{\mathcal{C}^1} e,
	\]
	implying in particular that
	\begin{align*}
		&\sum_{\mathcal{C} \in \Xi \colon \| \mathcal{C} \|_2 \leq 2(m-1)} \Psi_{\beta,\kappa}(\mathcal{C}) \pigl( \rho \bigl(\mathcal{C}^2(q_{\gamma_n+\gamma})\bigr) \mathbf{1}(\mathcal{C}^1 \nsim \gamma)-1\pigr) \cdot \mathbf{1}(\mathcal{C}^1 \neq \emptyset,\, n(\mathcal{C}^2)=1)
		\\&\qquad= 
		-2\xi_\beta \sum_{e \in C_1(B_N)^+} \sum_{\mathcal{C} \in \Xi^1 }  \Psi_{\kappa}(\mathcal{\mathcal{C}}) \deg_{\mathcal{C}} e \pigl( \bigl( 1-2 \mathbf{1}(e \in \gamma_n+\gamma)\bigr) \mathbf{1}(\mathcal{C} \nsim \gamma)-1\pigr).
	\end{align*}
	We now rewrite this sum as follows. First, note that 
	\begin{align}
		&  
		-2\xi_\beta \sum_{e \in C_1(B_N)^+} \sum_{\mathcal{C} \in \Xi^1 }  \Psi_{\kappa}(\mathcal{\mathcal{C}}) \deg_{\mathcal{C}} e \pigl( \bigl( 1-2 \mathbf{1}(e \in \gamma_n+\gamma)\bigr) \mathbf{1}(\mathcal{C} \nsim \gamma)-1\pigr)\nonumber
		\\&\qquad=
		-2\xi_\beta\sum_{\mathcal{C} \in \Xi^1 } \sum_{e \in \support \mathcal{C}}\Psi_\kappa(\mathcal{C})  \deg_\mathcal{C} e \pigl( \bigl(1-2\mathbf{1}(e \in \gamma_n+\gamma)\bigr)  \mathbf{1}(\mathcal{C} \nsim \gamma)-1\pigr) \nonumber 
		\\&\qquad=
		-2\xi_\beta\sum_{\mathcal{C} \in \Xi^1 } \sum_{e \in \support \mathcal{C}}\Psi_\kappa (\mathcal{C})  \deg_\mathcal{C} e \bigl( \mathbf{1}(\mathcal{C} \nsim \gamma)-1\bigr)  \label{eq: first sum to rewrite 3 new}
		\\&\qquad\qquad+4\xi_\beta
		\sum_{\mathcal{C} \in \Xi^1 } \sum_{e \in \support \mathcal{C}}\Psi_\kappa (\mathcal{C})  \deg_\mathcal{C} e  \mathbf{1}(e \in \gamma_n+\gamma) \mathbf{1}(\mathcal{C} \nsim \gamma). \label{eq: second sum to rewrite 3 new}
	\end{align}  
	where we further can rewrite~\eqref{eq: first sum to rewrite 3 new} as
	\begin{align}
		& 
		-2\xi_\beta\sum_{\mathcal{C} \in \Xi^1 } \sum_{e \in \support \mathcal{C}}\Psi_\kappa (\mathcal{C})  \deg_\mathcal{C} e \bigl( \mathbf{1}(\mathcal{C} \nsim \gamma)-1\bigr)  
		=
		-2\xi_\beta\sum_{\mathcal{C} \in \Xi^1 } \Psi_\kappa (\mathcal{C}) \| \mathcal{C}\| \bigl( \mathbf{1}(\mathcal{C} \nsim \gamma)-1\bigr)  .
		\label{eq: first sum to rewrite 3b new}
	\end{align}   
	For the sum in~\eqref{eq: second sum to rewrite 3 new}, we note that 
	\begin{align*}
		&4\xi_{\beta} \sum_{\mathcal{C} \in \Xi^1 } \sum_{e \in \support \mathcal{C}}\Psi_\kappa (\mathcal{C})  \deg_\mathcal{C} e  \mathbf{1}(e \in \gamma_n+\gamma) \mathbf{1}(\mathcal{C} \nsim \gamma)
		\\&\qquad=
		4\xi_{\beta} \sum_{\mathcal{C} \in \Xi^1 } \sum_{e \in \gamma_n}\Psi_\kappa (\mathcal{C})  \deg_\mathcal{C} e \mathbf{1}(\mathcal{C} \nsim \gamma) 
		=
		4\xi_{\beta} \sum_{\mathcal{C} \in \Xi^1 } \sum_{e \in \gamma_n}\Psi_\kappa(\mathcal{C})  \deg_\mathcal{C} e  
		+E_3(\gamma),
	\end{align*} 
	where \( E_3(\gamma) \) is defined in~\eqref{eq: error 4}.
	Combining the above equations, we obtain
	\begin{align*}
	 	&\sum_{\mathcal{C} \in \Xi} \Psi_{\beta,\kappa}(\mathcal{C}) \pigl( \rho(\mathcal{C}^2 \bigl(q_{\gamma_n+\gamma})\bigr) \mathbf{1}(\mathcal{C}^1 \nsim \gamma)-1\pigr)
	 	\\&\qquad=
		-2\xi_\beta|\gamma_n|-2\xi_\beta|\gamma| 
		+
		\sum_{\mathcal{C} \in \Xi^1} \Psi_{\kappa}(\mathcal{C})  \bigl( \mathbf{1}(\mathcal{C}^1 \nsim \gamma)-1\bigr) 
		-2\xi_\beta\sum_{\mathcal{C} \in \Xi^1 } \Psi_\kappa (\mathcal{C}) \| \mathcal{C}\| \bigl( \mathbf{1}(\mathcal{C} \nsim \gamma)-1\bigr) 
		\\&\qquad\qquad+4\xi_{\beta} \sum_{\mathcal{C} \in \Xi^1 } \sum_{e \in \gamma_n}\Psi_\kappa(\mathcal{C})  \deg_\mathcal{C} e  
	 	+
	 	E_1(\gamma)+E_2(\gamma)  +E_3(\gamma) , 
	\end{align*}
	and hence 
	\begin{align*}
	 	&\sum_{\gamma \in \Lambda^{\gamma_n}}  (\tanh 2\kappa)^{| \gamma|} e^{\sum_{\mathcal{C} \in \Xi} \Psi_{\beta,\kappa}(\mathcal{C}) \bigl( \rho(\mathcal{C}^2 (q_{\gamma_n+\gamma})) \mathbf{1}(\mathcal{C}^1 \nsim \gamma)-1\bigr)}
	 	\\&\qquad=
		e^{-2\xi_\beta|\gamma_n|}
		e^{
		+4\xi_{\beta} \sum_{\mathcal{C} \in \Xi^1 } \sum_{e \in \gamma_n}\Psi_\kappa(\mathcal{C})  \deg_\mathcal{C} e }
		\\&\qquad\quad\cdot 
		\sum_{\gamma \in \Lambda^{\gamma_n}}  (\tanh 2\kappa)^{| \gamma|} e^{\sum_{\mathcal{C} \in \Xi^1} \Psi_{\kappa}(\mathcal{C})  ( \mathbf{1}(\mathcal{C}^1 \nsim \gamma)-1)
		-2\xi_\beta|\gamma|   
		-2\xi_\beta\sum_{\mathcal{C} \in \Xi^1 } \Psi_\kappa (\mathcal{C}) \| \mathcal{C}\| ( \mathbf{1}(\mathcal{C} \nsim \gamma)-1)  
	 	+
	 	E_1(\gamma)+E_2(\gamma)  +E_3(\gamma)  }
	 	\\&\qquad=
		e^{-2\xi_\beta|\gamma_n|}
		e^{
		+4\xi_{\beta} \sum_{\mathcal{C} \in \Xi^1 } \sum_{e \in \gamma_n}\Psi_\kappa(\mathcal{C})  \deg_\mathcal{C} e }
		\\&\qquad\quad\cdot 
		\sum_{\gamma \in \Lambda^{\gamma_n}}   e^{ 
		-2\xi_\beta|\gamma|   
		-2\xi_\beta\sum_{\mathcal{C} \in \Xi^1 } \Psi_\kappa (\mathcal{C}) \| \mathcal{C}\| ( \mathbf{1}(\mathcal{C} \nsim \gamma)-1)  
	 	+
	 	E_1(\gamma)+E_2(\gamma)  +E_3(\gamma)  } \vartheta_{\gamma_n,\kappa}(\gamma).
	\end{align*}
	This completes the proof. 
\end{proof}

We now proceed to the proof of Theorem~\ref{theorem: dilute free phase}.
\begin{proof}[Proof of Theorem~\ref{theorem: dilute free phase}] 
	By Lemma~\ref{lemma: main contribution and errors new}, we have
	\begin{equation}\begin{split}\label{eq: step 1 of last proof}
		&\sum_{\gamma \in \Lambda^{\gamma_n}}  (\tanh 2\kappa)^{| \gamma|} e^{\sum_{\mathcal{C} \in \Xi} \Psi_{\beta_n,\kappa}(\mathcal{C}) ( \rho(\mathcal{C}^2(q_{\gamma_n+\gamma})) \mathbf{1}(\mathcal{C}^1 \nsim \gamma)-1)} 
		\\&\qquad=
		e^{-2\xi_{\beta_n}|\gamma_n|}
		e^{
		4\xi_{\beta_n} \sum_{\mathcal{C} \in \Xi^1 } \sum_{e \in \gamma_n}\Psi_\kappa(\mathcal{C})  \deg_\mathcal{C} e }
		\\&\qquad\quad\cdot 
		\sum_{\gamma \in \Lambda^{\gamma_n}}   e^{ 
		-2\xi_{\beta_n}|\gamma|   
		-2\xi_{\beta_n}\sum_{\mathcal{C} \in \Xi^1 } \Psi_\kappa (\mathcal{C}) \| \mathcal{C}\| ( \mathbf{1}(\mathcal{C} \nsim \gamma)-1)  
	 	+
	 	E_1(\gamma)+E_2(\gamma)  +E_3(\gamma)  } \vartheta_{\gamma_n,\kappa}(\gamma).
	\end{split}\end{equation} 
	Now note that
	\begin{align*}
		&\sum_{\gamma \in \Lambda^{\gamma_n}}   e^{ 
		-2\xi_{\beta_n}|\gamma|   
		-2\xi_{\beta_n}\sum_{\mathcal{C} \in \Xi^1 } \Psi_\kappa (\mathcal{C}) \| \mathcal{C}\| ( \mathbf{1}(\mathcal{C} \nsim \gamma)-1)  
	 	+
	 	E_1(\gamma)+E_2(\gamma)  +E_3(\gamma)  } \vartheta_{\gamma_n,\kappa}(\gamma)
	 	\\&\qquad=
	 	\sum_{\gamma \in \Lambda^{\gamma_n}}   e^{ 
		-2\xi_{\beta_n}|\gamma|   
		-2\xi_{\beta_n}\sum_{\mathcal{C} \in \Xi^1 } \Psi_\kappa (\mathcal{C}) \| \mathcal{C}\| ( \mathbf{1}(\mathcal{C} \nsim \gamma)-1)    } \vartheta_{\gamma_n,\kappa}(\gamma)
	 	\\&\qquad\qquad+
	 	\sum_{\gamma \in \Lambda^{\gamma_n}}   e^{ 
		-2\xi_{\beta_n}|\gamma|   
		-2\xi_{\beta_n}\sum_{\mathcal{C} \in \Xi^1 } \Psi_\kappa (\mathcal{C}) \| \mathcal{C}\| ( \mathbf{1}(\mathcal{C} \nsim \gamma)-1)  
	 	} \bigl(e^{
	 	E_1(\gamma)+E_2(\gamma)  +E_3(\gamma)  } -1 \bigr)\vartheta_{\gamma_n,\kappa}(\gamma). 
	\end{align*}
	Then, by the AM-GM inequality, we have
	\begin{align*}
		&0 \leq \sum_{\gamma \in \Lambda^{\gamma_n}}   e^{ 
		-2\xi_{\beta_n}|\gamma|   
		-2\xi_{\beta_n}\sum_{\mathcal{C} \in \Xi^1 } \Psi_\kappa (\mathcal{C}) \| \mathcal{C}\| ( \mathbf{1}(\mathcal{C} \nsim \gamma)-1)  
	 	} \bigl(e^{
	 	E_1(\gamma)+E_2(\gamma)  +E_3(\gamma)  } -1 \bigr)\vartheta_{\gamma_n,\kappa}(\gamma)
	 	\\&\qquad\leq 
	 	\sum_{\gamma \in \Lambda^{\gamma_n}}   e^{ 
		-2\xi_{\beta_n}|\gamma|   
		-2\xi_{\beta_n}\sum_{\mathcal{C} \in \Xi^1 } \Psi_\kappa (\mathcal{C}) \| \mathcal{C}\| ( \mathbf{1}(\mathcal{C} \nsim \gamma)-1)  
	 	} \max_{j \in \{ 1,2,4\}}\bigl(e^{
	 	3 E_j(\gamma) }-1 \bigr)\vartheta_{\gamma_n,\kappa}(\gamma)
	\end{align*}
	%
	Using Lemma~\ref{lemma: E1}, Lemma~\ref{lemma: E2}, Lemma~\ref{lemma: E4}, and Lemma~\ref{lemma: last decay}, it thus follows that 
	\begin{equation}\begin{split}\label{eq: step 2 of last proof}
		&\sum_{\gamma \in \Lambda^{\gamma_n}}   e^{ 
		-2\xi_{\beta_n}|\gamma|   
		-2\xi_{\beta_n}\sum_{\mathcal{C} \in \Xi^1 } \Psi_\kappa (\mathcal{C}) \| \mathcal{C}\| ( \mathbf{1}(\mathcal{C} \nsim \gamma)-1)  
	 	+
	 	E_1(\gamma)+E_2(\gamma)  +E_3(\gamma)  } \vartheta_{\gamma_n,\kappa}(\gamma)
	 	\\&\qquad=
	 	\sum_{\gamma \in \Lambda^{\gamma_n}}   e^{ 
		-2\xi_{\beta_n}|\gamma|   
		-2\xi_{\beta_n}\sum_{\mathcal{C} \in \Xi^1 } \Psi_\kappa (\mathcal{C}) \| \mathcal{C}\| ( \mathbf{1}(\mathcal{C} \nsim \gamma)-1)     } \vartheta_{\gamma_n,\kappa}(\gamma)
	 	+ 
	 	\vartheta_{\gamma_n,\kappa}(\Lambda^{\gamma_n})o_n(1).
	\end{split}\end{equation}

	Now recall the definition of \( D_1 \) from~\eqref{eq: def Dm}, and note that
	\begin{equation*}
		\Bigl| \sum_{\mathcal{C} \in \Xi^1 } \Psi_\kappa (\mathcal{C}) \| \mathcal{C}\| \bigl( \mathbf{1}(\mathcal{C} \nsim \gamma)-1\bigr)  \Bigr| 
		\leq 
		2 |\gamma|\sup_{e \in C_1(B_N)} \sum_{\mathcal{C} \in \Xi^1_e } \bigl| \Psi_\kappa (\mathcal{C}) \bigr| \| \mathcal{C}\|  = 2|\gamma|D_1(\kappa).
	\end{equation*}  
	Now assume that \( \kappa \) is sufficiently small to ensure that \( 2D_1(\kappa)<1, \) and let \( \varepsilon>0 \) be such that \( 2D_1(\kappa)<1-\varepsilon. \) Then
	\begin{equation*}
		-2\xi_{\beta_n}|\gamma|   
		-2\xi_{\beta_n}\sum_{\mathcal{C} \in \Xi^1 } \Psi_\kappa (\mathcal{C}) \| \mathcal{C}\| \bigl( \mathbf{1}(\mathcal{C} \nsim \gamma)-1\bigr)  < -2\varepsilon \xi_{\beta_n}|\gamma|<0,
	\end{equation*}
	and hence 
	\begin{align*}
	 	&
		\sum_{\gamma \in \Lambda^{\gamma_n}}   e^{ 
		-2\xi_{\beta_n}|\gamma|   
		-2\xi_{\beta_n}\sum_{\mathcal{C} \in \Xi^1 } \Psi_\kappa (\mathcal{C}) \| \mathcal{C}\| ( \mathbf{1}(\mathcal{C} \nsim \gamma)-1)     } \vartheta_{\gamma_n,\kappa}(\gamma) 
		>  
		\sum_{\gamma \in \Lambda^{\gamma_n}}   e^{ 
		-2\varepsilon\xi_{\beta_n}|\gamma|    } \vartheta_{\gamma_n,\kappa}(\gamma) .
	 \end{align*} 
	Using Lemma~\ref{lemma: lower bound}, we obtain
	 \begin{equation}\label{eq: new goal}
	 	\liminf_{n \to \infty}
	 	\vartheta_{\gamma_n,\kappa}(\Lambda^{\gamma_n})^{-1}\sum_{\gamma \in \Lambda^{\gamma_n}}   e^{ 
		-2\xi_{\beta_n}|\gamma|   
		-2\xi_{\beta_n}\sum_{\mathcal{C} \in \Xi^1 } \Psi_\kappa (\mathcal{C}) \| \mathcal{C}\| ( \mathbf{1}(\mathcal{C} \nsim \gamma)-1)     } \vartheta_{\gamma_n,\kappa}(\gamma) >0.  
	 \end{equation} 
	 Letting
	 \begin{equation}\label{eq: C def}\begin{split}
	 	&C_{\beta_n,\kappa} \coloneqq e^{-2\xi_{\beta_n}|\gamma_n| + 4\xi_{{\beta_n}} |\gamma_n| \sum_{\mathcal{C} \in \Xi^1_e }  \Psi_\kappa(\mathcal{C})  \deg_\mathcal{C} e }
		\\&\qquad\qquad\cdot 
		\vartheta_{\gamma_n,\kappa}(\Lambda^{\gamma_n})^{-1} \sum_{\gamma \in \Lambda^{\gamma_n}}   e^{ 
		-2\xi_{\beta_n}|\gamma|   
		-2\xi_{\beta_n}\sum_{\mathcal{C} \in \Xi^1 } \Psi_\kappa (\mathcal{C}) \| \mathcal{C}\| ( \mathbf{1}(\mathcal{C} \nsim \gamma)-1)     } \vartheta_{\gamma_n,\kappa}(\gamma).
	 \end{split}\end{equation}
	 and combining~\eqref{eq: step 1 of last proof}, \eqref{eq: step 2 of last proof} and~\eqref{eq: new goal}, it follows that 
	 \begin{equation*}\begin{split}
		&\sum_{\gamma \in \Lambda^{\gamma_n}}  (\tanh 2\kappa)^{| \gamma|} e^{\sum_{\mathcal{C} \in \Xi} \Psi_{\beta_n,\kappa}(\mathcal{C}) ( \rho(\mathcal{C}^2(q_{\gamma_n+\gamma})) \mathbf{1}(\mathcal{C}^1 \nsim \gamma)-1)} 
		=
		C_{\beta_n,\kappa} \vartheta_{\gamma_n,\kappa}(\Lambda^{\gamma_n}) \bigl(1+o_n(1)\bigr). 	\end{split}\end{equation*} 
	 Recalling~\eqref{eq: the expansion} and~\eqref{eq: ising decay measure}, the desired conclusion immediately follows. 
\end{proof}

\end{document}